\documentclass[reqno]{amsart}

\usepackage[top=30truemm,bottom=30truemm,left=25truemm,right=25truemm]{geometry}
\usepackage{amsmath}
\usepackage{amssymb}
\usepackage{amsfonts}
\usepackage{amsthm}
\usepackage{graphicx}
\usepackage{wrapfig}
\usepackage{enumerate}
\usepackage{physics}

\theoremstyle{plain}
 \newtheorem{thm}{Theorem}
 \newtheorem{lem}[thm]{Lemma}
 \newtheorem{cor}[thm]{Corollary}
 \newtheorem{prop}[thm]{Proposition}
 
\theoremstyle{definition}

\newcommand{\thmref}[1]{Theorem~\ref{#1}}
\newcommand{\lemref}[1]{Lemma~\ref{#1}}
\newcommand{\cororef}[1]{Corollary~\ref{#1}}
\newcommand{\propref}[1]{Proposition~\ref{#1}}

\begin{document}

\title[Arithmetical independence of certain uniform sets of algebraic integers]{Arithmetical independence of certain uniform sets of\\ algebraic integers}

\author[Asaki Saito]{Asaki Saito$^{\ast}$}
\address{Asaki Saito, Department of Complex and Intelligent Systems, Future University Hakodate, 116-2 Kamedanakano-cho, Hakodate, Hokkaido 041-8655, Japan}
\email{saito@fun.ac.jp}
\thanks{$^\ast$Corresponding author.}

\author{Jun-ichi Tamura}
\address{Jun-ichi Tamura, Institute for Mathematics and Computer Science, Tsuda College, 2-1-1 Tsuda-machi, Kodaira, Tokyo 187-8577, Japan}
\email{jtamura@tsuda.ac.jp}

\author{Shin-ichi Yasutomi}
\address{Shin-ichi Yasutomi, Faculty of Science, Toho University, 2-2-1 Miyama, Funabashi, Chiba 274-8510, Japan}
\email{shinichi.yasutomi@sci.toho-u.ac.jp}

\begin{abstract}
We study four (families of) sets of algebraic integers of degree less
than or equal to three.
Apart from being simply defined, we show that they share two
distinctive characteristics: almost uniformity and arithmetical independence.
Here, ``almost uniformity'' means that the elements of a finite
set are distributed almost equidistantly in the unit interval,
while ``arithmetical independence'' means that the number
fields generated by the elements of a set do not have a mutual inclusion
relation each
other.
Furthermore, we reveal to what extent the algebraic number fields
generated by the elements of the four sets can cover quadratic or cubic
fields.
\end{abstract}

\date{\today}

\maketitle

\section{Introduction}

Saito and Yamaguchi \cite{SaitoChaos2016, SaitoChaos2018} have
proposed to use the binary sequence $\left\{ \epsilon_i
\right\}_{i=1,2,\dots}$ ($\epsilon_i \in \left\{0,\,1\right\}$)
obtained from the binary expansion $\alpha = \sum_{i=1}^{\infty}
\epsilon_i 2^{-i}$ of an irrational algebraic integer $\alpha$ in the
open unit interval $(0,1)$ as a pseudorandom binary sequence.
As a set of seeds $\alpha$ for pseudorandom number generation, they introduced two sets, $I_{n}^{2,r}$
and $I_{m,n}^{3,ntr}$:
For any integer $n$ satisfying either $n \ge 1$ or $n \le -3$, a set
$I_{n}^{2,r}$ of real algebraic integers of degree two over ${\mathbb
  Q}$ is defined by
\begin{align}\label{eq:RealQuadraticI}
I_{n}^{2,r} := \left\{
\begin{array}{ll}
\left\{\alpha \in (0,1) \mid \alpha^2+n\alpha+c=0, c \in {\mathbb Z}, -n \le c \le -1\right\} &   \textrm{~  if~} n \ge 1, \\
\left\{\alpha \in (0,1) \mid \alpha^2+n\alpha+c=0, c \in {\mathbb Z}, 1 \le c \le -n-2\right\} &   \textrm{~  if~} n \le -3.
\end{array}\right.
\end{align}
Likewise, for any integers $m, n$ satisfying $m^2-3n \le 0$ and $m+n
\ge 1$, a set $I_{m,n}^{3,ntr}$ of real algebraic integers of degree
three over ${\mathbb Q}$ is defined by
\begin{align}\label{eq:NotTotallyRealCubicDefinition}
I_{m,n}^{3,ntr} &:= \left\{\alpha \in (0,1) \mid \alpha^3+m\alpha^2+n\alpha+d=0, d \in {\mathbb Z}, -(m+n) \le d \le -1\right\}.
\end{align}
Note that any $\alpha \in I_{m,n}^{3,ntr}$ is not totally real, that is,
$\alpha$ has two complex conjugates (see \cite{SaitoChaos2018}).
Both $I_{n}^{2,r}$ and $I_{m,n}^{3,ntr}$ have some desirable properties as a
set of seeds, such as almost uniform distribution of their elements
in the unit interval.
For definition of almost uniformity, see end of this section.
In particular, they have been observed to have a remarkable property ---arithmetical independence--- from computer experiments \cite{SaitoChaos2016, SaitoChaos2018}.
Here, we say that a set $S \subset {\mathbb C}$ is {\it arithmetically
  independent} if $\alpha \notin {\mathbb Q}(\beta)$ and $\beta \notin
{\mathbb Q}(\alpha)$ hold for any $\alpha \neq \beta \in S$.
If $\alpha$, $\beta$ are algebraic numbers of prime degrees, then
${\mathbb Q}(\alpha) \neq {\mathbb Q}(\beta)$ is equivalent to $\alpha
\notin {\mathbb Q}(\beta)$ and $\beta \notin {\mathbb Q}(\alpha)$.

When generating multiple pseudorandom sequences, it is desirable that
they are as different from each other as possible.
For example, it is undesirable that a certain binary sequence $\left\{
\epsilon_i \right\}_{i=1,2,\dots}$ and its complement $\left\{
1-\epsilon_i \right\}_{i=1,2,\dots}$ are generated simultaneously.
One can guarantee that such a case does not occur if
a seed set $I$ is arithmetically independent.
(Note that $\alpha \in I$ implies $1-\alpha \notin I$.)
However, the arithmetical independence of
$I_{n}^{2,r}$ or $I_{m,n}^{3,ntr}$ was only presented as conjectures
and no proofs have been given yet.
Moreover, algebraic and number theoretical properties, other than the
arithmetical independence, of $I_{n}^{2,r}$ and $I_{m,n}^{3,ntr}$
still remain to be elucidated.

In this paper in addition to the above-mentioned sets $I_{n}^{2,r}$
and $I_{m,n}^{3,ntr}$, we introduce two sets, $I_{n}^{2,i}$ and
$I_{m,n}^{3,tr}$:
A set $I_{n}^{2,i}$ of imaginary algebraic integers of degree two is
defined for any integer $n \ge 1$ as follows:
\begin{align}\label{eq:ImaginaryQuadraticDefinition}
I_{n}^{2,i} := \left\{
\begin{array}{ll}
\left\{\displaystyle\frac{1+\sqrt{1-4c}}{2} \in {\mathbb C} \setminus {\mathbb R} ~\middle|~  c \in {\mathbb Z}, \left(\displaystyle\frac{n-1}{2}\right)^2 +1\le c \le \left(\displaystyle\frac{n+1}{2}\right)^2\right\} &   \textrm{~  if~} n \textrm{~is odd}, \\
\left\{\displaystyle\sqrt{-c} \in {\mathbb C} \setminus {\mathbb R} ~\middle|~  c \in {\mathbb Z}, \left(\displaystyle\frac{n}{2}\right)^2 < c < \left(\displaystyle\frac{n}{2} +1\right)^2\right\} &   \textrm{~  if~} n \textrm{~is even}.
\end{array}\right.
\end{align}
Note that $\left(1+\sqrt{1-4c}\right)/2$ is a root of the polynomial $X^2
- X + c$ and $\sqrt{-c}$ is that of $X^2 + c$.
Likewise, a set $I_{m,n}^{3,tr}$ of real algebraic integers of degree
at most three is defined for any integers $m, n$ satisfying $n \le
-m-3$ as follows:
\begin{align}\label{eq:TotallyRealCubicDefinition}
I_{m,n}^{3,tr} &:= \left\{\alpha \in (0,1) \mid \alpha^3+m\alpha^2+n\alpha+d=0, d \in {\mathbb Z}, 1 \le d \le -m-n-2\right\}.
\end{align}
For these four (families of) sets of algebraic integers, we show that
they share two distinctive characteristics: the almost uniformity and
the arithmetical independence.
Moreover, we show that certain unions of $I_{n}^{2,r}$ (resp. $I_{n}^{2,i}$,
$I_{m,n}^{3,ntr}$, and $I_{m,n}^{3,tr}$) are {\it full generative} with respect to
real quadratic fields (resp. imaginary
quadratic fields except
${\mathbb Q}(\sqrt{-1})$, real cubic fields which are not totally real, and totally real cubic fields)
in the sense
that it contains a primitive element of any real quadratic field (resp. imaginary
quadratic field except
${\mathbb Q}(\sqrt{-1})$, real cubic field which is not totally real, and totally real cubic field).

To the best of our knowledge,
$I_{n}^{2,r}$, $I_{n}^{2,i}$, $I_{m,n}^{3,ntr}$, and $I_{m,n}^{3,tr}$ are the
simplest sets guaranteed to be arithmetically independent,
consisting of algebraic integers of degree two or three.
In number theory, it is often the case that the difference in
discriminants is used to indicate the difference in fields.
However, this is hardly simple as a way of defining sets;
in particular it requires further number-theoretic discussion
when dealing with fields of degree three or higher with the same
discriminant (for details on handling the cubic case, see, e.g.,
\cite{DeloneFaddeev,HambletonSAWilliamHC}).
Apart from sets defined using $n$th roots of appropriately chosen
integers, which typically have elements with larger heights than
$I_{n}^{2,r}$, $I_{n}^{2,i}$, $I_{m,n}^{3,ntr}$, and $I_{m,n}^{3,tr}$
(for background on heights, see, e.g., \cite{Silverman}), there seem to be almost
no other sets of algebraic numbers, or families of algebraic number
fields, known to be arithmetically independent.
The exception is the one-parameter family of totally real cubic fields
$K_{a}$ generated by $x^3 - a x^2 -(a+3) x -1$ with $a \in {\mathbb
  Z}$.
The field $K_{a}$ is called a simplest cubic field \cite{Shanks}, and
necessary and sufficient conditions are given for $K_{a}$ and $K_{b}$
with $a \neq b$ to be isomorphic \cite{Morton}.
(The problem of determining whether two fields belonging to a specific
family are isomorphic is sometimes called field isomorphism problem;
see, e.g., \cite{Hoshi}.)
Okazaki showed that all $K_{a}$ in this family are
distinct from each other, except for the obvious coincidence of $K_{a}
= K_{-a-3}$, and a few exceptions: $K_{-1} = K_{5} = K_{12} =
K_{1259}$, $K_{0} = K_{3} = K_{54}$, $K_{1} = K_{66}$, and $K_{2} =
K_{2389}$ for $a \ge -1$.
Hoshi gave another proof of this result \cite[Theorem 1.4]{Hoshi}, and
also showed a similar result for a one-parameter family of totally
real sextic fields \cite{Hoshi2}, as well as for a one-parameter family of
totally real fields of degree 24 and 12 \cite{Hoshi3}.

Before proceeding
we explain what finite sets in the closed unit interval
$[0,1]$ are called uniform or almost uniform in this paper.
Let $N$ be a positive integer.
Let $S = S_N = \left\{x_i=x_i^{(N)} ~\middle|~ i=1,2,\dots, N\right\}$ be a finite set with
$N$ elements in $[0,1]$.
Without loss of generality, we may assume that the elements are ordered increasingly, i.e., $0 \le
x_1 \le x_2 \le \dots \le x_N \le 1$.
Let $\Delta_{i} := x_{i+1} - x_{i}$ ($i \in \left\{1,2,\dots,
N-1\right\}$) be a distance between two neighboring elements.

We say that the elements of a finite set $S \subset [0,1]$ are
distributed {\it uniformly} in the unit interval if
$\Delta_{i} = 1/N$ holds for all $i \in \left\{1,2,\dots, N-1\right\}$.
This definition of uniformity is based on the minimality of the
discrepancy $D_N$, given by
\[
D_N = \sup_{0 \le a < b \le 1} \left | \frac{A\left( [a,b); S\right)}{N}
  - (b-a) \right |,
\]
where $A\left( [a,b); S\right)$ is the number of elements in $S$ that
fall into $[a,b)$ \cite{Niederreiter}.
Using \cite[theorem 2.7]{Niederreiter}, we see that
a finite sequence $\left\{y_1,y_2,\dots,y_N\right\}$
in $[0,1]$ gives the minimum value $1/N$ of 
$D_N$ if and only if there exists $\delta \in [0,1/N]$ such that
$y_1,y_2,\dots,y_N$ are the numbers
$1/N-\delta,2/N-\delta,\dots,N/N-\delta$ in some order.

Next we explain the concept of almost uniformity.
Let $S_1, S_2, \dots$ be a family of finite sets in $[0,1]$.
Assume that there exists a constant $c$ independent of $N$ such
that
\[
\max_{i \in \left\{1,2,\dots, N-1\right\}} \left| \Delta_{i} -\frac{1}{N} \right|
\le \frac{c}{N^2}
\]
holds for all $N$.
We say that the elements of $S_N$ are distributed {\it almost
  uniformly} in the unit interval, or simply that $S_N$ is almost
uniform, if $N \ge 2 c$.
Note that one can impose a condition on $N$ stricter than $N \geq 2c$
in order to avoid less uniform $S_N$, but such differences in
conditions do not affect the subsequent discussion.
For simplicity, without explicitly specifying the constant $c$ (and
thus the condition $N \geq 2c$), we say that the elements of $S_N$ are
distributed almost uniformly in the unit interval for sufficiently
large $N$.
Obviously, $S_N$ is almost uniform for sufficiently large $N$
if
\begin{equation}\label{eq:MaxOrderInequality}
\max_{i \in \left\{1,2,\dots, N-1\right\}} \left| \Delta_{i} -\frac{1}{N} \right|
= O\left(\frac{1}{N^2}\right)
\end{equation}
holds, where $O$ is Landau's Big O notation.
We can see from \cite[theorem 2.7]{Niederreiter} that the discrepancy
$D_N$ of $S_N$ satisfying \eqref{eq:MaxOrderInequality} satisfies $D_N
= O\left(1/N\right)$.

\section{Properties of $I_{n}^{2,r}$}
\label{sec:PropertiesofI_{n}^{2,r}}

In this section, we investigate properties of the sets $I_{n}^{2,r}$
of real algebraic integers of degree two, which are defined in
\eqref{eq:RealQuadraticI}.
Concerning the uniformity of $I_{n}^{2,r}$ in the unit interval,
we have the following result.

\begin{prop}[\cite{SaitoChaos2016}]\label{Prop:QuadraticUniformity}
The elements of $I_{b}^{2,r}$
are distributed almost uniformly in the unit interval for sufficiently
large $|b|$.
\end{prop}

In terms of pseudorandom number generation, this characteristic of
$I_{n}^{2,r}$ is desirable for unbiased sampling of seeds.

As mentioned in the introduction, however, another characteristic of
$I_{n}^{2,r}$, namely the arithmetical independence, was only presented as a conjecture in \cite{SaitoChaos2016}.
Here, we give a proof of the conjecture.

\begin{thm}\label{Thm:QuadraticNoncoincidence}
Let $b$ be an integer other than $0$, $-1$, and $-2$. Then, ${\mathbb
Q}(\alpha) \neq {\mathbb Q}(\beta)$ holds for all $\alpha, \beta
\in I_{b}^{2,r}$ with $\alpha \neq \beta$.
\end{thm}

\begin{proof}
Assume that there are $\alpha, \beta \in I_{b}^{2,r}$ such that
$\alpha \neq \beta$ and
${\mathbb Q}(\alpha) = {\mathbb Q}(\beta)$.
Set $K:={\mathbb Q}(\alpha) = {\mathbb Q}(\beta)$.
Let $\alpha_1$ and $\alpha_2$ be the conjugates of $\alpha$.
We may assume that $\alpha_1 = \alpha$.
Then, $\alpha_2 = -b -\alpha_1$.
Similarly, let $\beta_1$ and $\beta_2$ be the conjugates of $\beta$,
and assume that $\beta_1 = \beta$ and $\beta_2 = -b -\beta_1$.
Then, we have
\[
\left|\alpha_2 - \beta_2\right| = \left|\alpha_1 - \beta_1\right| < 1,
\]
where the inequality follows from $\alpha_1, \beta_1 \in (0,1)$.
The above gives
\begin{equation}\label{eq:RealQuadraticInequality}
\left|\left(\alpha_1 - \beta_1\right)\left(\alpha_2 - \beta_2\right)\right| < 1.
\end{equation}
Since $\left(\alpha_1 - \beta_1\right)\left(\alpha_2 - \beta_2\right)$
is the norm of $\alpha_1 - \beta_1$ as an element of $K$, it is a
rational number, but it is obviously an algebraic integer.
Thus, $\left(\alpha_1 - \beta_1\right)\left(\alpha_2 -
\beta_2\right)$ is an integer.
From \eqref{eq:RealQuadraticInequality}, we obtain
\[
\left(\alpha_1 - \beta_1\right)\left(\alpha_2 - \beta_2\right)=0,
\]
which implies $\alpha=\beta$.
This contradicts the assumption $\alpha \neq \beta$.
\end{proof}

We now see the diversity of the quadratic fields
generated by the elements in
$I_{n}^{2,r}$ from another viewpoint.
Let $S^{2,r}$ be the union of $I_{n}^{2,r}$ with positive $n$,
namely
\begin{equation}\label{eq:S2r}
S^{2,r}:=\bigcup_{n \ge 1} I_{n}^{2,r}.
\end{equation}
It turns out that $S^{2,r}$ includes a primitive element of any real
quadratic field, that is, $S^{2,r}$ is full generative with respect to
real quadratic fields.
We can see this from two stronger properties of $S^{2,r}$ and
$I_{n}^{2,r}$ (see Theorems \ref{Thm:RealQuadraticCover} and
\ref{Thm:RealQuadraticVarietyI} below).

In preparation for the next statement, we introduce some notation.
Let $S$ be a set of complex numbers.
We put
\begin{align*}
  +S := S,\qquad
  -S := \left\{ -\alpha ~\vert~ \alpha \in S \right\}.
\end{align*}
For an integer $n$, we further put
\begin{align*}
  S + n := \left\{ \alpha + n ~\vert~ \alpha \in S \right\},\qquad
  S - n := \left\{ \alpha - n ~\vert~ \alpha \in S \right\}.
\end{align*}
We put $\left\langle x \right\rangle := x - \lfloor x \rfloor$, where
$\lfloor x \rfloor$ is the greatest integer which does not
exceed a real number $x$.
In what follows, we call an algebraic integer of degree two over
${\mathbb Q}$ a {quadratic algebraic integer}.
We also need the following equality:
\begin{align}\label{eq:RealQuadraticSymmetry}
I_{-n-2}^{2,r} &= \left\{1-\alpha \mid \alpha \in I_{n}^{2,r}\right\}.
\end{align}
This equality follows from the fact that $\alpha$ is a root of
$X^2+bX+c$ if and only if $1 - \alpha$ is a root of $X^2+(-b-2)X
+b+c+1$.

\begin{thm}\label{Thm:RealQuadraticCover}
The followings hold:
\begin{enumerate}[(i)]
\item The set of all quadratic algebraic integers in ${\mathbb R}$
  is expressed as
\begin{equation}\label{eq:RealQuadraticDisjointUnion}
\biguplus_{\substack{n \in {\mathbb Z}\\\varepsilon \in \left\{-1,+1\right\}}} \left(\varepsilon S^{2,r} + n \right),
\end{equation}
where $\biguplus$ denotes disjoint union.

\item In particular, $S^{2,r}$ includes a primitive element of any real quadratic field.
\end{enumerate}
\end{thm}

\begin{proof}
We prove (i).
We know that the set of all quadratic algebraic integers in the open
unit interval $(0,1)$ is given by
\begin{equation*}
\bigcup_{n \in {\mathbb Z} \setminus \left\{0, -1, -2\right\}} I_{n}^{2,r}
\end{equation*}
(see \cite{SaitoChaos2016}).
Using \eqref{eq:RealQuadraticSymmetry}, we express this set as
$S^{2,r} \cup \left(-S^{2,r} + 1 \right)$.
Let now $\alpha$ be an arbitrary quadratic algebraic integer in
${\mathbb R}$.
Then, $\left\langle \alpha
\right\rangle$ is contained in $S^{2,r} \cup
\left(-S^{2,r} + 1 \right)$.
This implies
\[
\alpha \in \bigcup_{n \in {\mathbb Z}} \left[\left\{ S^{2,r} \cup
  \left(-S^{2,r} + 1 \right) \right\} + n \right].
\]
It is easy to see that the family $\left\{ S^{2,r} + n ~\vert~ n \in
{\mathbb Z} \right\} \cup \left\{ -S^{2,r} + n ~\vert~ n \in {\mathbb
  Z} \right\}$ is mutually disjoint.

To prove (ii), it suffices, by (i), to note that any real quadratic field,
denoted by $K$, includes real quadratic algebraic integers, any of
which is a primitive element of $K$.
\end{proof}

Moreover, we can show that for any given finite combination of real quadratic
fields, there exists $I_{n}^{2,r}$ some of whose elements can generate
the given fields over ${\mathbb Q}$:

\begin{thm}\label{Thm:RealQuadraticVarietyI}
Let $k$ be a positive integer.
Let $K_1, K_2, \dots, K_k$ be distinct real quadratic fields.
Then, there exists a positive integer $n$ such that $I_{2 n}^{2,r}$
contains a primitive element of every $K_i$ ($i \in \{1,2,\dots,k\}$).
\end{thm}

\begin{proof}
For each $K_i$, there exists a positive integer $j_i$ such that $K_i =
{\mathbb Q}(\sqrt{j_i})$ ($i \in \{1,2,\dots,k\}$).
Clearly, $1, \sqrt{j_1}, \dots, \sqrt{j_k}$ are linearly independent
over ${\mathbb Q}$, and so are
$
1, -\frac{1}{\sqrt{j_1}}, \dots, -\frac{1}{\sqrt{j_k}}.
$
Let $\epsilon$ be a positive number such that $\epsilon \sqrt{j_i} <
1$ for all $i \in \{1,2,\dots,k\}$.
By Kronecker's theorem \cite{Hardy}, there exist integers
$n>0, m_1, m_2, \dots, m_k$ such that
\[
0<-\frac{n}{\sqrt{j_i}} + m_i<\epsilon \qquad (i \in \{1,2,\dots,k\}).
\]
Thus, we have
$
0<-n  + m_i \sqrt{j_i} <\epsilon \sqrt{j_i} < 1$ $(i \in \{1,2,\dots,k\}).
$
Set $\alpha_i := -n  + m_i \sqrt{j_i}$ ($i \in \{1,2,\dots,k\}$).
Obviously, $\alpha_i$ is a quadratic algebraic integer in $(0,1)$, as
well as a primitive element of $K_i$.
Let $\alpha_i^{\sigma}$ be the conjugate of $\alpha_i$ that is
different from $\alpha_i$, namely $\alpha_i^{\sigma} = -n - m_i \sqrt{j_i}$.
Since $\alpha_i + \alpha_i^{\sigma} = -2n$, we see that all $\alpha_i$
($i \in \{1,2,\dots,k\}$) are included in $I_{2 n}^{2,r}$.
\end{proof}

We note here the following point.
One can obtain, in a trivial manner, a family of sets of quadratic
algebraic integers in $(0,1)$, which has the properties described in
Theorems \ref{Thm:QuadraticNoncoincidence} and
\ref{Thm:RealQuadraticVarietyI}.
One example is the family $I_{n}^{2,p}$ ($n \in {\mathbb Z}_{>0}$)
defined by
\[
I_{n}^{2,p} := \left\{ \left\langle \sqrt{p_1} \right\rangle,
\left\langle \sqrt{p_2} \right\rangle, \dots, \left\langle \sqrt{p_n} \right\rangle\right\},
\]
where $p_1, p_2, \dots, p_n$ are the first $n$ prime numbers.
The uniformity of $I_{n}^{2,p}$, however, is inferior to that of
$I_{n}^{2,r}$.
For example, $I_{8}^{2,p}$ consists of five elements less than $1/2$ and
three elements greater than $1/2$.
This is in contrast with the fact that for $I_{n}^{2,r}$ with even
$n$, the number of elements less than $1/2$ is the same as the number
of elements greater than $1/2$, and for $I_{n}^{2,r}$ with odd positive (resp. negative) $n$,
the number of elements less than $1/2$ is exactly one less (resp. more) than the
number of elements greater than $1/2$.
Likewise, there is no expression involving $S^{2,p} :=\bigcup_{n \ge
  1} I_{n}^{2,p}$, which is as simple as
\eqref{eq:RealQuadraticDisjointUnion} and which can cover the set of
all quadratic algebraic integers in ${\mathbb R}$.
Thus, what makes $I_{n}^{2,r}$ distinctive is not only
the properties described in
Theorems \ref{Thm:QuadraticNoncoincidence} and
\ref{Thm:RealQuadraticVarietyI}
but also other properties such as
those described in \propref{Prop:QuadraticUniformity} and
\thmref{Thm:RealQuadraticCover} (i), in addition to its simple definition.
In the next section we will see that $I_{n}^{2,i}$ occupies an
equivalent position in the imaginary quadratic domain.

In preparation for the next section, we give the following lemma.

\begin{lem}\label{Lem:RealQuadraticShiftedI}
Let $n$ be a positive integer.
The followings hold:
\begin{enumerate}[(i)]
\item For even $n$, we have
\[
I_{n}^{2,r} + \frac{n}{2} = \left\{\displaystyle\sqrt{-c} \in {\mathbb R} ~\middle|~  c \in {\mathbb Z}, -\left(\displaystyle\frac{n}{2} +1\right)^2 < c < -\left(\displaystyle\frac{n}{2}\right)^2\right\}.
\]

\item For odd $n$, we have
\begin{align*}
& I_{n}^{2,r} + \frac{n+1}{2}\\ &=\left\{\displaystyle\frac{1+\sqrt{1-4c}}{2} \in {\mathbb R} ~\middle|~  c \in {\mathbb Z}, -\left(\displaystyle\frac{n-1}{2}+1\right) \left(\displaystyle\frac{n-1}{2}+2\right) < c < -\displaystyle\frac{n-1}{2} \left(\displaystyle\frac{n-1}{2}+1\right) \right\}.
\end{align*}

\end{enumerate}
\end{lem}

\begin{proof}
We prove (i).
Since $n>0$, $\alpha \in (0,1)$ satisfying $\alpha^2+n\alpha+c=0$ is
given by $\alpha = (-n+\sqrt{n^2-4c})/2$
(cf. \eqref{eq:RealQuadraticI}).
Thus, $\alpha + n/2$ can be written as $\alpha + n/2 = \sqrt{-c'}$,
where $c'= - (n/2)^2 + c$.
As easily verified, $c \in {\mathbb Z}$ satisfies $-n \le c \le -1$ if and
only if $c' \in {\mathbb Z}$ satisfies
$-\left(n/2 +1\right)^2 < c' < -\left(n/2\right)^2$.

We can prove (ii) in the same way as (i).
\end{proof}

\section{Properties of $I_{n}^{2,i}$}

In this section, we consider the imaginary quadratic case, which is
the counterpart of the real quadratic case considered in Section
\ref{sec:PropertiesofI_{n}^{2,r}}.
For $z = x + \sqrt{-1} y$ ($x,y \in {\mathbb R}$), we put $\Re(z) :=
x$, $\Im(z) := y$, and $\bar{z} := x - \sqrt{-1} y$, as usual.
For a subset $S$ of ${\mathbb C} \setminus {\mathbb R}$, we denote by
$\Im(S)$ the set $\left\{ \Im(\alpha) ~\vert~ \alpha \in S \right\}$,
and by $\left\langle \Im(S) \right\rangle$ the set $\left\{
\left\langle \Im(\alpha) \right\rangle ~\vert~ \alpha \in S \right\}$.
In this paper, we identify the uniformity of a set $S \subset
{\mathbb C} \setminus {\mathbb R}$ as the uniformity of the set
$\left\langle \Im(S) \right\rangle$ in the unit interval.

We obtain the following result concerning the uniformity of
$I_{n}^{2,i}$ (see \eqref{eq:ImaginaryQuadraticDefinition} for the
definition of $I_{n}^{2,i}$).

\begin{thm}\label{Thm:ImaginaryQuadraticUniformity}
The elements of $\left\langle \Im(I_{n}^{2,i}) \right\rangle$
are distributed almost uniformly in the
unit interval for sufficiently large $n$.
\end{thm}

\begin{proof}
For even $n$, it follows from \eqref{eq:ImaginaryQuadraticDefinition} and
\lemref{Lem:RealQuadraticShiftedI} (i) that
\begin{equation}\label{eq:SetImAlpha}
\Im(I_{n}^{2,i}) = I_{n}^{2,r} + \frac{n}{2}.
\end{equation}
The uniformity of $I_{n}^{2,i}$ then follows from
\propref{Prop:QuadraticUniformity}.

For odd $n$, we see that
$\Im\left(\left(1+\sqrt{1-4c}\right)/2\right)=\sqrt{c-1/4}$.
Let $m:= (n-1)/2$.
Since $m^2+1 \le c \le (m+1)^2$, we see that $m < \sqrt{c-1/4} < m+1$.
We assume from now on that $n \neq 1$.
The distances $\Delta_{i}$ ($i \in \left\{1,2,\dots, 2 m\right\}$)
between two neighboring elements of $\left\langle \Im(I_{n}^{2,i})
\right\rangle$ are given by
$\Delta_{i} = \sqrt{(m^2+ i +1) -1/4} - \sqrt{(m^2+ i) -1/4}$.
By the mean value theorem, there exists $\gamma_{i} \in \left(m^2+ i, m^2+
i +1 \right)$ such that $\Delta_{i} = \left(2
\sqrt{\gamma_{i}-1/4}\right)^{-1}$.
Thus, we have
\[
\frac{1}{2 m +2}
< \frac{1}{2\sqrt{(m^2+ 2 m +1)-1/4}} < \Delta_{i} < \frac{1}{2\sqrt{(m^2+1)-1/4}}
< \frac{1}{2 m}
\]
for every $i \in \left\{1,2,\dots, 2 m\right\}$,
so that
\[
\frac{1}{n+1}  - \frac{1}{n}
<  \Delta_{i} - \frac{1}{n}
< \frac{1}{n-1}  - \frac{1}{n},
\]
which implies
\[
\max_{i \in \left\{1,2,\dots, n-1\right\}} \left| \Delta_{i} -\frac{1}{n} \right|
< \frac{1}{n (n-1)}.
\]
Thus, the elements of $\left\langle \Im(I_{n}^{2,i})
\right\rangle$ are
distributed in the unit interval almost uniformly for sufficiently large $n$.
\end{proof}

It follows from \eqref{eq:SetImAlpha} together with the fact stated before
\lemref{Lem:RealQuadraticShiftedI} that for even $n$,
$\left\langle \Im(I_{n}^{2,i}) \right\rangle$
contains an equal number of elements less and
greater than $1/2$.
Also for odd $n$, it is easy to verify that the number of elements less
than $1/2$ is exactly one less than the number of elements greater
than $1/2$, similarly as in $I_{n}^{2,r}$ with $n \ge 1$.

Next we show that $I_{n}^{2,i}$ is also arithmetically independent.

\begin{thm}\label{Thm:ImaginaryQuadraticNoncoincidence}
Let $n$ be a positive integer. Then, ${\mathbb
Q}(\alpha) \neq {\mathbb Q}(\beta)$ holds for all $\alpha, \beta
\in I_{n}^{2,i}$ with $\alpha \neq \beta$.
\end{thm}

\begin{proof}
We first consider the case where $n$ is even.
We see that ${\mathbb Q}(\sqrt{-c}) \neq {\mathbb Q}(\sqrt{-c'})$
holds if $c$ and $c'$ are positive integers satisfying ${\mathbb
  Q}(\sqrt{c}) \neq {\mathbb Q}(\sqrt{c'})$.
Thus, the assertion for even $n$ follows from
\thmref{Thm:QuadraticNoncoincidence} together with
\eqref{eq:ImaginaryQuadraticDefinition} and
\lemref{Lem:RealQuadraticShiftedI} (i).

Now consider the case where $n$ is odd.
Assume that there are $\alpha, \beta \in I_{n}^{2,i}$ such that
$\alpha \neq \beta$ and
${\mathbb Q}(\alpha) = {\mathbb Q}(\beta)$.
Set $K:={\mathbb Q}(\alpha) = {\mathbb Q}(\beta)$.
From the proof of \thmref{Thm:ImaginaryQuadraticUniformity} we have
\begin{align*}
\alpha = \frac{1}{2} + \left(m + \left\langle
\Im(\alpha) \right\rangle \right)\sqrt{-1},\qquad
\beta = \frac{1}{2} + \left(m + \left\langle
\Im(\beta) \right\rangle \right)\sqrt{-1},
\end{align*}
where $m= (n-1)/2$.
This gives
$
\left|\alpha - \beta\right|=\left|\left\langle
\Im(\alpha) \right\rangle - \left\langle
\Im(\beta) \right\rangle \right| < 1
$, so that
$
\left|\left(\alpha - \beta\right) \left(\overline{\alpha - \beta}\right)\right| < 1.
$
Since $\left(\alpha - \beta\right) \left(\overline{\alpha -
  \beta}\right)$ is an algebraic integer as well as the norm of
$\alpha - \beta$ as an element of $K$, it is an integer.
Thus, $\left(\alpha - \beta\right) \left(\overline{\alpha -
  \beta}\right) = 0$, which implies $\alpha=\beta$.
Since $\alpha \neq \beta$, we get a contradiction.
\end{proof}

We now turn to consider to what extent the union of $I_{n}^{2,i}$ or a
single $I_{n}^{2,i}$ can cover the imaginary quadratic fields in the sense of
Theorems \ref{Thm:RealQuadraticCover} and
\ref{Thm:RealQuadraticVarietyI}.
We let
\[
S^{2,i}:=\bigcup_{n \ge 1} I_{n}^{2,i}.
\]
Since $4c-1$ with $c \in {\mathbb Z}_{>0}$ cannot be a square number,
we have $S^{2,i} \cap {\mathbb Q}(\sqrt{-1}) = \emptyset$.
Set
\begin{equation}\label{eq:S2iHat}
\hat{S}^{2,i}:= S^{2,i} \cup \left\{ n \sqrt{-1}  ~\vert~ n \in {\mathbb Z}_{>0} \right\}.
\end{equation}
We will see below that $\hat{S}^{2,i}$ thus defined is full generative
with respect to imaginary quadratic fields, and therefore $S^{2,i}$ is
full generative with respect to imaginary quadratic fields except
${\mathbb Q}(\sqrt{-1})$ (see Theorems
\ref{Thm:ImaginaryQuadraticCover} and
\ref{Thm:ImaginaryQuadraticVarietyI}).

To simplify the exposition below, we set
\begin{align*}
  (b,c)_+ := (-b+\sqrt{b^2-4c})/2,\qquad
  (b,c)_- := (-b-\sqrt{b^2-4c})/2.
\end{align*}
From \eqref{eq:ImaginaryQuadraticDefinition} and \eqref{eq:S2iHat}, we
see that
\begin{align*}
  \hat{S}^{2,i} = \left\{ (b,c)_+  ~\vert~ b \in \left\{0,-1\right\}, c \in {\mathbb Z}_{>0} \right\},\qquad
  -\hat{S}^{2,i} = \left\{ (b,c)_-  ~\vert~ b \in \left\{0,1\right\}, c \in {\mathbb Z}_{>0} \right\}.
\end{align*}
The set of all quadratic algebraic integers in ${\mathbb C}
\setminus {\mathbb R}$ is given by
\[
\left\{ (b,c)_\pm  ~\vert~ (b, c) \in {\mathbb Z}^2,  c> b^2/4 \right\}.
\]
An easy calculation shows that for any integer $n$, we have
\[
(b,c)_\pm + n = (-2n + b, n^2 -b n +c)_\pm.
\]

\begin{thm}\label{Thm:ImaginaryQuadraticCover}
The followings hold:
\begin{enumerate}[(i)]
\item The set of all quadratic algebraic integers in ${\mathbb C}
\setminus {\mathbb R}$
  is expressed as
\begin{equation*}%\label{eq:RealQuadraticDisjointUnion}
\biguplus_{\substack{n \in {\mathbb Z}\\\varepsilon \in \left\{-1,+1\right\}}} \left(\varepsilon \hat{S}^{2,i} + n \right).
\end{equation*}
The set of all quadratic algebraic integers in ${\mathbb C}
\setminus {\mathbb R}$ except those in ${\mathbb Q}(\sqrt{-1})$
is expressed as
\begin{equation*}%\label{eq:RealQuadraticDisjointUnion}
\biguplus_{\substack{n \in {\mathbb Z}\\\varepsilon \in \left\{-1,+1\right\}}} \left(\varepsilon S^{2,i} + n \right).
\end{equation*}

\item In particular, $\hat{S}^{2,i}$ includes a primitive element of
any imaginary quadratic field;
$S^{2,i}$ includes a primitive element of any imaginary
quadratic field except ${\mathbb Q}(\sqrt{-1})$.
\end{enumerate}
\end{thm}

\begin{proof}
We have
\begin{align*}
\hat{S}^{2,i} + n &= \left( \left\{ (0,c)_+  ~\vert~ c \in {\mathbb Z}_{>0} \right\} + n \right) \bigcup \left( \left\{ (-1,c)_+  ~\vert~ c \in {\mathbb Z}_{>0} \right\} + n \right),\\
 &= \left\{ (-2n, n^2 +c)_+  ~\vert~ c \in {\mathbb Z}_{>0} \right\} \bigcup \left\{ (-2n -1, n^2 +n +c)_+  ~\vert~ c \in {\mathbb Z}_{>0} \right\}.
\end{align*}
Considering that $(-2n)^2/4 = n^2$, $(-2n -1)^2/4 = n^2 +n +1/4$, and
$\cup_{n \in {\mathbb Z}} \{-2n, -2n -1\}={\mathbb Z}$, we obtain
\[
\bigcup_{n \in {\mathbb Z}} \left(\hat{S}^{2,i} + n \right)
= \left\{ (b,c)_+  ~\vert~ (b, c) \in {\mathbb Z}^2,  c> b^2/4 \right\}.
\]
By a similar argument, we obtain
\[
\bigcup_{n \in {\mathbb Z}} \left(-\hat{S}^{2,i} + n \right)
= \left\{ (b,c)_-  ~\vert~ (b, c) \in {\mathbb Z}^2,  c> b^2/4 \right\}.
\]
Therefore
\begin{align*}
  \left\{ (b,c)_\pm  ~\vert~ (b, c) \in {\mathbb Z}^2,  c> b^2/4 \right\} &=
  \left( \bigcup_{n \in {\mathbb Z}} \left(\hat{S}^{2,i} + n \right) \right)
  \bigcup
  \left( \bigcup_{n \in {\mathbb Z}} \left(-\hat{S}^{2,i} + n \right) \right).
\end{align*}
It is easy to see that the family $\left\{ \hat{S}^{2,i} + n ~\vert~ n \in
{\mathbb Z} \right\} \cup \left\{ -\hat{S}^{2,i} + n ~\vert~ n \in {\mathbb
  Z} \right\}$ is mutually disjoint.
The second statement of (i) is derived from the first statement by
considering the fact that for $n \in {\mathbb Z}$, $(b,c)_\pm + n \in
{\mathbb Q}(\sqrt{-1})$ if and only if $(b,c)_\pm \in {\mathbb
  Q}(\sqrt{-1})$.

As in the proof of \thmref{Thm:RealQuadraticCover}, (ii) follows from
(i).
\end{proof}

The following theorem is the imaginary counterpart of
\thmref{Thm:RealQuadraticVarietyI}.

\begin{thm}\label{Thm:ImaginaryQuadraticVarietyI}
Let $k$ be a positive integer.
Let $K_1, K_2, \dots, K_k$ be distinct imaginary quadratic fields
other than ${\mathbb Q}(\sqrt{-1})$.
Then, there exists a positive integer $n$ such that $I_{2 n}^{2,i}$
contains a primitive element of every $K_i$ ($i \in \{1,2,\dots,k\}$).
\end{thm}

\begin{proof}
For each $K_i$, there exists a square-free integer $j_i \ge 2$ such that $K_i =
{\mathbb Q}(\sqrt{- j_i})$ ($i \in \{1,2,\dots,k\}$).
We define the real quadratic fields $K_i'$ ($i \in \{1,2,\dots,k\}$)
by $K_i' := {\mathbb Q}(\sqrt{j_i})$.
By \thmref{Thm:RealQuadraticVarietyI}, there exists a positive integer
$n$ such that $I_{2 n}^{2,r}$ contains a primitive element of every
$K_i'$.
Obviously, $I_{2 n}^{2,r} + n$ also contains a primitive element of
every $K_i'$.
Using \lemref{Lem:RealQuadraticShiftedI} (i), we express this set
as
\[
I_{2 n}^{2,r} + n = \left\{\displaystyle\sqrt{c} \in {\mathbb R} ~\middle|~  c \in {\mathbb Z}, n^2 < c < \left(n +1\right)^2 \right\}.
\]
Thus, for each $j_i$, there exists an integer $c_i$ in the range $n^2
< c_i < \left(n +1\right)^2$ such that $c_i = m_i^2 j_i$ for some
integer $m_i$.
By \eqref{eq:ImaginaryQuadraticDefinition} we have
\[
I_{2 n}^{2,i} = \left\{\displaystyle\sqrt{-c} \in {\mathbb C} \setminus {\mathbb R} ~\middle|~  c \in {\mathbb Z}, n^2 < c < \left(n +1\right)^2 \right\},
\]
and we find that $\sqrt{-c_i} \in I_{2 n}^{2,i}$ is a primitive
element of $K_i = {\mathbb Q}(\sqrt{- j_i})$.
\end{proof}

Lastly in this section, it is worth noting that as in the real
quadratic case, one can define a set $\tilde{I}_{n}^{2,i}$ which is
twin to $I_{n}^{2,i}$ by
\begin{align}\label{eq:ImaginaryQuadraticSymmetry}
\tilde{I}_{n}^{2,i} &:= \left\{1-\alpha \mid \alpha \in I_{n}^{2,i}\right\}
\qquad n = 1, 2, \dots,
\end{align}
see \eqref{eq:RealQuadraticSymmetry}.
By using \eqref{eq:ImaginaryQuadraticDefinition}, we obtain explicit
formulae for $\tilde{I}_{n}^{2,i}$:
\begin{align*}%\label{eq:ImaginaryQuadraticTwinDefinition}
\tilde{I}_{n}^{2,i} = \left\{
\begin{array}{ll}
\left\{ (-1,c)_- \in {\mathbb C} \setminus {\mathbb R} ~\middle|~  c \in {\mathbb Z}, \left(\displaystyle\frac{n-1}{2}\right)^2 +1\le c \le \left(\displaystyle\frac{n+1}{2}\right)^2\right\} &   \textrm{~  if~} n \textrm{~is odd}, \\
\left\{ (-2,c)_- \in {\mathbb C} \setminus {\mathbb R} ~\middle|~  c \in {\mathbb Z}, \left(\displaystyle\frac{n}{2}\right)^2 +1 < c < \left(\displaystyle\frac{n}{2} +1\right)^2 +1 \right\} &   \textrm{~  if~} n \textrm{~is even}.
\end{array}\right.
\end{align*}
Due to definition \eqref{eq:ImaginaryQuadraticSymmetry},
$\tilde{I}_{n}^{2,i}$ has the properties described in Theorems
\ref{Thm:ImaginaryQuadraticNoncoincidence} and
\ref{Thm:ImaginaryQuadraticVarietyI}.
Also the uniformity of $\tilde{I}_{n}^{2,i}$ is as good as that of
$I_{n}^{2,i}$
since
\[
\left\langle \Im(\tilde{I}_{n}^{2,i}) \right\rangle = \left\{1-\beta \mid \beta \in \left\langle \Im(I_{n}^{2,i}) \right\rangle \right\}
\]
holds.

\section{Properties of $I_{m,n}^{3,ntr}$}

In this section, we investigate properties of the sets
$I_{m,n}^{3,ntr}$ of real cubic algebraic integers that are not
totally real (see \eqref{eq:NotTotallyRealCubicDefinition} for the
definition of $I_{m,n}^{3,ntr}$).
As for the uniformity of $I_{m,n}^{3,ntr}$ in the unit interval,
we have the following result.

\begin{prop}[\cite{SaitoChaos2018}]\label{Prop:CubicUniformity}
Let $b$ be a fixed integer.
Then, the elements of $I_{b,c}^{3,ntr}$ are distributed almost uniformly
in the unit interval for sufficiently large $c$.
\end{prop}

Moreover, we can easily see that similarly to $I_{n}^{2,r}$ ($n \ge
1$) and $I_{n}^{2,i}$ ($n \ge 1$) discussed in the previous sections,
the following holds for $I_{m,n}^{3,ntr}$ with $m \in
\left\{0,\,-1\right\}$, namely $I_{0,n}^{3,ntr}$ ($n \ge 1$) and
$I_{-1,n}^{3,ntr}$ ($n \ge 2$):
If $I_{0,n}^{3,ntr}$ (or $I_{-1,n}^{3,ntr}$) consists of even number of
elements, then the elements less than $1/2$ and those greater than
$1/2$ are equal in number.
If $I_{0,n}^{3,ntr}$ (or $I_{-1,n}^{3,ntr}$) consists of odd number of
elements, then the number of elements less than $1/2$ is exactly one
less than the number of elements greater than $1/2$.
From the fact that $\alpha$ is a root of $X^3+bX^2+cX+d$ if and only
if $1 - \alpha$ is a root of $X^3+(-b-3)X^2+(2b + c +3)X +(-b-c-d-1)$,
we obtain
\begin{align}\label{eq:RealCubicSymmetry}
I_{-b-3, 2b + c +3}^{3,ntr} &= \left\{1-\alpha \mid \alpha \in I_{b,c}^{3,ntr}\right\}.
\end{align}
Thus, we see that the elements of $I_{-2,n}^{3,ntr}$ ($n \ge 3$) (or
$I_{-3,n}^{3,ntr}$ ($n \ge 4$)) are also divided into the intervals
$(0,1/2)$ and $(1/2,1)$ as equally as possible (i.e., the number of
elements in $(0,1/2)$ is equal to or exactly one more than the number of
elements in $(1/2,1)$).
On the other hand, for $m$ other than $0$, $-1$, $-2$, and $-3$,
uniformity in this sense is inferior.
In fact, it is not difficult to see that even if $I_{m,n}^{3,ntr}$ with
$m \in {\mathbb Z} \setminus \left\{0, -1, -2, -3\right\}$ consists of
even number of elements, the elements less than $1/2$ and those
greater than $1/2$ are not equal in number (i.e., the difference
between the number of elements in $(0,1/2)$ and the number of elements
in $(1/2,1)$ is greater than or equal to two).
Thus, we will specialise to $I_{m,n}^{3,ntr}$ with $m \in \left\{0, -1,
-2, -3\right\}$ below.

We now show the arithmetical independence of
$I_{m,n}^{3,ntr}$ with $m \in \left\{0, -1, -2, -3\right\}$, which was
originally conjectured in \cite[Conjecture 1]{SaitoChaos2018}.

\begin{thm}\label{Thm:CubicNoncoincidence}
Let $b, c$ be integers satisfying $-3 \le b \le 0$ and $c \ge -b+1$.
Then, ${\mathbb Q}(\alpha) \neq {\mathbb Q}(\beta)$ holds for all
$\alpha, \beta \in I_{b,c}^{3,ntr}$ with $\alpha \neq \beta$.
\end{thm}

\begin{proof}
Assume that $\alpha, \beta \in I_{b,c}^{3,ntr}$ with $\alpha \neq
\beta$ satisfy ${\mathbb Q}(\alpha) = {\mathbb Q}(\beta)$.
Let $\alpha_i$ ($i \in \{1,2,3\}$) (resp. $\beta_i$ ($i \in \{1,2,3\}$))
be the conjugates of $\alpha$ (resp. $\beta$).
We may assume that $\alpha_1 = \alpha$, $\beta_1 = \beta$, and
${\mathbb Q}(\alpha_i)={\mathbb Q}(\beta_i)$ ($i \in \{1,2,3\}$).
Note that $\alpha_i$ and $\beta_i$ with $i > 1$ are imaginary numbers,
and that $\alpha_3$ (resp. $\beta_3$) is the complex conjugate of
$\alpha_2$ (resp. $\beta_2$).
We assume that $\alpha_2$ is in the upper half of the complex plane
($\beta_2$ can be in the upper or lower half plane).

We divide the proof into four cases:
(i) $b=0$, $c \ge 1$,
(ii) $b=-1$, $c \ge 2$,
(iii) $b=-2$, $c \ge 3$,
(iv) $b=-3$, $c \ge 4$.

\noindent
{\bf Case (i):}
For calculational convenience, we temporarily
assume that $c \ge 10$.
By the relation between roots and coefficients, we have
\begin{align*}
\alpha_1 + \alpha_2 + \alpha_3 = 0,\qquad
\alpha_1 \alpha_2 + \alpha_2 \alpha_3+ \alpha_3 \alpha_1 = c.
\end{align*}
Thus, we have
\begin{align*}
  \alpha_1 (\alpha_2 + \alpha_3) + \alpha_2 \alpha_3 = c,\qquad
  \alpha_2 \alpha_3 = c + \alpha_1^2,\qquad
  \lvert \alpha_2 \rvert (=\lvert \alpha_3 \rvert) = \sqrt{c + \alpha_1^2}.
\end{align*}
Therefore, we obtain
$
\alpha_2 = \displaystyle -\frac{\alpha_1}{2} + \sqrt{c + \frac{3 \alpha_1^2}{4}} i,
$
where we denote for simplicity $\sqrt{-1}$ by $i$.
In a similar manner, we have (P)
$
\beta_2 = \displaystyle -\frac{\beta_1}{2} + \sqrt{c + \frac{3 \beta_1^2}{4}} i
$ or (M)
$
\beta_2 = \displaystyle -\frac{\beta_1}{2} - \sqrt{c + \frac{3 \beta_1^2}{4}} i
$.

First, assume (P).
Then we have
\[
\alpha_2 - \beta_2 = -\frac{\alpha_1 - \beta_1}{2} + \left( \sqrt{c + \frac{3 \alpha_1^2}{4}} - \sqrt{c + \frac{3 \beta_1^2}{4}} \right) i,
\]
which means that
\begin{align*}
  \lvert \alpha_2 - \beta_2 \rvert &\le \frac{\lvert \alpha_1 - \beta_1 \rvert}{2} + \left\lvert \sqrt{c + \frac{3 \alpha_1^2}{4}} - \sqrt{c + \frac{3 \beta_1^2}{4}} \right\rvert\\
&= \frac{\lvert \alpha_1 - \beta_1 \rvert}{2} +
  \left\lvert \frac{\displaystyle\frac{3 (\alpha_1^2 - \beta_1^2)}{4}}{\displaystyle\sqrt{c + \frac{3 \alpha_1^2}{4}} + \sqrt{c + \frac{3 \beta_1^2}{4}}} \right\rvert\\
&< \displaystyle\frac{1}{2} +   \frac{\displaystyle\frac{3}{4}}{2 \displaystyle\sqrt{c}} \le \frac{1}{2} +   \frac{3}{8 \displaystyle\sqrt{10}} < 1.
\end{align*}  
Since $\alpha_3 - \beta_3 = \overline{\alpha_2 - \beta_2}$, we also have
$\lvert \alpha_3 - \beta_3 \rvert < 1$.
Thus, we see that $\lvert (\alpha_1 - \beta_1)(\alpha_2 - \beta_2)(\alpha_3 - \beta_3) \rvert < 1$.
Since $(\alpha_1 - \beta_1)(\alpha_2 - \beta_2)(\alpha_3 - \beta_3)
\in {\mathbb Z}$, we have $(\alpha_1 - \beta_1)(\alpha_2 -
\beta_2)(\alpha_3 - \beta_3) =0$, which implies $\alpha_1 = \beta_1$.
This contradicts $\alpha \neq \beta$.

Second, assume (M).
Then we have
\[
\alpha_2 + \beta_2 = -\frac{\alpha_1 + \beta_1}{2} + \left( \sqrt{c + \frac{3 \alpha_1^2}{4}} - \sqrt{c + \frac{3 \beta_1^2}{4}} \right) i,
\]
which means that
\begin{align*}
  \lvert \alpha_2 + \beta_2 \rvert &\le  \frac{\left\lvert\alpha_1 + \beta_1\right\rvert}{2}  + \left\lvert \sqrt{c + \frac{3 \alpha_1^2}{4}} - \sqrt{c + \frac{3 \beta_1^2}{4}} \right\rvert\\
  & < 1 +   \frac{\displaystyle\frac{3}{4}}{2 \displaystyle\sqrt{c}} \le 1 +   \frac{3}{8 \displaystyle\sqrt{10}} < 1.1186.
\end{align*}  
We also have
$\lvert \alpha_3 + \beta_3 \rvert < 1.1186$.
Since $\lvert \alpha_1 + \beta_1 \rvert < 2$, we see that $\lvert
(\alpha_1 + \beta_1)(\alpha_2 + \beta_2)(\alpha_3 + \beta_3) \rvert <
2 \times (1.1186)^2 =2.50253192$.
Since $(\alpha_1 + \beta_1)(\alpha_2 + \beta_2)(\alpha_3 + \beta_3)
\in {\mathbb Z}$, we see that $(\alpha_1 + \beta_1)(\alpha_2 +
\beta_2)(\alpha_3 + \beta_3)$ is equal to either 0, 1, or 2.
Assuming $(\alpha_1 + \beta_1)(\alpha_2 + \beta_2)(\alpha_3 + \beta_3)=0$,
we have $\alpha_1 + \beta_1=0$, which is a contradiction.
Thus, $(\alpha_1 + \beta_1)(\alpha_2 + \beta_2)(\alpha_3 + \beta_3)$
is equal to 1 or 2.
We now show that $\alpha_1 + \beta_1 \notin {\mathbb Q}$ (i.e.,
$\alpha_1 + \beta_1 \notin {\mathbb Z}$).
Assuming $r:=\alpha_1 + \beta_1 \in {\mathbb Q}$, we have $3 r=
(\alpha_1 + \beta_1)+(\alpha_2 + \beta_2)+(\alpha_3 + \beta_3)=0$,
which contradicts $r \neq 0$.
Thus, $\alpha_1 + \beta_1$ is a cubic algebraic integer.
We have
\begin{align*}
&\lvert (\alpha_1 + \beta_1)(\alpha_2 + \beta_2)+ (\alpha_2 + \beta_2)(\alpha_3 + \beta_3)+ (\alpha_3 + \beta_3)(\alpha_1 + \beta_1) \rvert\\
& \le \lvert \alpha_1 + \beta_1 \rvert \lvert \alpha_2 + \beta_2 \rvert+ \lvert \alpha_2 + \beta_2 \rvert \lvert \alpha_3 + \beta_3 \rvert+ \lvert \alpha_3 + \beta_3 \rvert \lvert \alpha_1 + \beta_1 \rvert\\
  & < 2 \times 1.1186 + (1.1186)^2 + 1.1186 \times 2\\
  & = 5.72566596 <6.
\end{align*}
Therefore, the minimal polynomial of $\alpha_1 + \beta_1$, denoted by
$x^3 + p x - q$, satisfies $\lvert p \rvert \le 5$ and $q \in \{1,2\}$.
We numerically found the roots of $x^3 + p x - q$ which is
irreducible and whose discriminant is negative, and we show the
results in Table~\ref{tab:b=0}.
\begin{table}
\caption{\label{tab:b=0}
Numerical roots of irreducible polynomials $x^3 + p x - q$ with $\lvert p \rvert \le 5$, $q \in \{1,2\}$, and negative discriminant}
%\begin{ruledtabular}
\begin{center}
\begin{tabular}{lrr}
Polynomial &
Real root &
Imaginary roots\\
\hline
$x^3-2 x-2$ & $1.76929$ & $-0.88465 \pm 0.58974 i$\\
$x^3-x-1$   & $1.32472$ & $-0.66236 \pm 0.56228 i$\\
$x^3-x-2$   & $1.52138$ & $-0.76069 \pm 0.85787 i$\\
$x^3-2$     & $1.25992$ & $-0.62996 \pm 1.09112 i$\\
$x^3+x-1$   & $0.68233$ & $-0.34116 \pm 1.16154 i$\\
$x^3+2 x-1$ & $0.45340$ & $-0.22670 \pm 1.46771 i$\\
$x^3+2 x-2$ & $0.77092$ & $-0.38546 \pm 1.56388 i$\\
$x^3+3 x-1$ & $0.32219$ & $-0.16109 \pm 1.75438 i$\\
$x^3+3 x-2$ & $0.59607$ & $-0.29804 \pm 1.80734 i$\\
$x^3+4 x-1$ & $0.24627$ & $-0.12313 \pm 2.01134 i$\\
$x^3+4 x-2$ & $0.47347$ & $-0.23673 \pm 2.04160 i$\\
$x^3+5 x-1$ & $0.19844$ & $-0.09922 \pm 2.24266 i$\\
$x^3+5 x-2$ & $0.38829$ & $-0.19415 \pm 2.26121 i$
\end{tabular}
\end{center}
%\end{ruledtabular}
\end{table}
We can see from the above argument that
the absolute value of the imaginary part of
$\alpha_2 + \beta_2$ is less than $0.1186$, but there is no polynomial
having such an imaginary root in Table~\ref{tab:b=0}.
So again we arrive at a contradiction.

Consequently, we have proved the assertion for $c \ge 10$ in Case (i).
The assertion for $1 \le c < 10$ has already been
confirmed by direct computation in \cite{SaitoChaos2018}.

\noindent
{\bf Case (ii):}
As in the proof of Case (i), we temporarily assume that $c \ge 10$.
Using
\begin{align*}
\alpha_1 + \alpha_2 + \alpha_3 = 1,\qquad
\alpha_1 \alpha_2 + \alpha_2 \alpha_3+ \alpha_3 \alpha_1 = c,
\end{align*}
we have
\begin{align*}
  \alpha_1 (\alpha_2 + \alpha_3) + \alpha_2 \alpha_3 = c,\qquad
  \alpha_2 \alpha_3 = c - \alpha_1 (1 - \alpha_1),\qquad
  \lvert \alpha_2 \rvert (=\lvert \alpha_3 \rvert) = \sqrt{c - \alpha_1 (1 - \alpha_1)},
\end{align*}
and we obtain
$
\alpha_2 = \displaystyle\frac{1 - \alpha_1}{2} + \sqrt{c + \frac{3 \alpha_1^2 - 2\alpha_1 -1}{4}} i.
$
Similarly, we have (P)
$
\beta_2 = \displaystyle\frac{1 - \beta_1}{2} + \sqrt{c + \frac{3 \beta_1^2 - 2\beta_1 -1}{4}} i
$ or (M)
$
\beta_2 = \displaystyle\frac{1 - \beta_1}{2} - \sqrt{c + \frac{3 \beta_1^2 - 2\beta_1 -1}{4}} i
$.

First, assume (P).
Then we have
\[
\alpha_2 - \beta_2 = -\frac{\alpha_1 - \beta_1}{2} + \left( \sqrt{c + \frac{3 \alpha_1^2 - 2\alpha_1 -1}{4}} - \sqrt{c + \frac{3 \beta_1^2 - 2\beta_1 -1}{4}} \right) i,
\]
which means that
\begin{align*}
  \lvert \alpha_2 - \beta_2 \rvert &\le \frac{\lvert \alpha_1 - \beta_1 \rvert}{2} + \left\lvert \sqrt{c + \frac{3 \alpha_1^2 - 2\alpha_1 -1}{4}} - \sqrt{c + \frac{3 \beta_1^2 - 2\beta_1 -1}{4}} \right\rvert\\
    &= \frac{\lvert \alpha_1 - \beta_1 \rvert}{2} +
  \left\lvert \frac{\displaystyle\frac{(\alpha_1 - \beta_1) \left\{3 (\alpha_1 + \beta_1) -2\right\} }{4}}{\displaystyle\sqrt{c + \frac{3 \alpha_1^2 - 2\alpha_1 -1}{4}} + \sqrt{c + \frac{3 \beta_1^2 - 2\beta_1 -1}{4}}} \right\rvert.
\end{align*}  
Since $3 x^2 - 2x -1 = 3 (x - 1/3)^2 - 4/3$, we have
\[
\lvert \alpha_2 - \beta_2 \rvert < \frac{1}{2} +   \frac{1}{2 \displaystyle\sqrt{c - \frac{1}{3}}} \le \frac{1}{2} +   \frac{1}{2 \displaystyle\sqrt{\frac{29}{3}}} < 1.
\]
We also have
$\lvert \alpha_3 - \beta_3 \rvert < 1$.
Thus, we see that $\lvert (\alpha_1 - \beta_1)(\alpha_2 - \beta_2)(\alpha_3 - \beta_3) \rvert < 1$,
and we have $(\alpha_1 - \beta_1)(\alpha_2 -
\beta_2)(\alpha_3 - \beta_3) =0$.
This contradicts $\alpha \neq \beta$.

Second, assume (M).
Then we have
\[
\alpha_2 + \beta_2 = 1-\frac{\alpha_1 + \beta_1}{2} + \left( \sqrt{c + \frac{3 \alpha_1^2 - 2\alpha_1 -1}{4}} - \sqrt{c + \frac{3 \beta_1^2 - 2\beta_1 -1}{4}} \right) i,
\]
which means that
\begin{align*}
  \lvert \alpha_2 + \beta_2 \rvert &\le \left\lvert 1-\frac{\alpha_1 + \beta_1}{2} \right\rvert + \left\lvert \sqrt{c + \frac{3 \alpha_1^2 - 2\alpha_1 -1}{4}} - \sqrt{c + \frac{3 \beta_1^2 - 2\beta_1 -1}{4}} \right\rvert\\
  & < 1 +   \frac{1}{2 \displaystyle\sqrt{c - \frac{1}{3}}} \le 1 +   \frac{1}{2 \displaystyle\sqrt{\frac{29}{3}}} < 1.1609.
\end{align*}  
We also have
$\lvert \alpha_3 + \beta_3 \rvert < 1.1609$.
Thus, $\lvert
(\alpha_1 + \beta_1)(\alpha_2 + \beta_2)(\alpha_3 + \beta_3) \rvert <
2 \times (1.1609)^2 =2.69537762$.
Since $\alpha_1 + \beta_1>0$, we have $(\alpha_1 + \beta_1)(\alpha_2 +
\beta_2)(\alpha_3 + \beta_3) > 0$.
Thus, $(\alpha_1 + \beta_1)(\alpha_2 + \beta_2)(\alpha_3 + \beta_3)$
is equal to 1 or 2.
We now show that $\alpha_1 + \beta_1 \notin {\mathbb Q}$, namely
$\alpha_1 + \beta_1 \notin {\mathbb Z}$.
Assuming $z:=\alpha_1 + \beta_1 \in {\mathbb Z}$, we have $3 z=
(\alpha_1 + \beta_1)+(\alpha_2 + \beta_2)+(\alpha_3 + \beta_3)=2$,
which gives the contradiction $z=2/3$.
Thus, $\alpha_1 + \beta_1$ is a cubic algebraic integer.
We have
\begin{align*}
&\lvert (\alpha_1 + \beta_1)(\alpha_2 + \beta_2)+ (\alpha_2 + \beta_2)(\alpha_3 + \beta_3)+ (\alpha_3 + \beta_3)(\alpha_1 + \beta_1) \rvert\\
& \le \lvert \alpha_1 + \beta_1 \rvert \lvert \alpha_2 + \beta_2 \rvert+ \lvert \alpha_2 + \beta_2 \rvert \lvert \alpha_3 + \beta_3 \rvert+ \lvert \alpha_3 + \beta_3 \rvert \lvert \alpha_1 + \beta_1 \rvert\\
  & < 2 \times 1.1609 + (1.1609)^2 + 1.1609 \times 2\\
  & = 5.99128881 <6.
\end{align*}
Therefore, the minimal polynomial of $\alpha_1 + \beta_1$, denoted by
$x^3 -2 x^2 + p x - q$, satisfies $\lvert p \rvert \le 5$ and $q \in \{1,2\}$.
Table~\ref{tab:b=-1} shows numerical roots of $x^3 -2 x^2 + p x - q$
which is irreducible and whose discriminant is negative.
\begin{table}
\caption{\label{tab:b=-1}
Numerical roots of irreducible polynomials $x^3 -2 x^2 + p x - q$ with $\lvert p \rvert \le 5$, $q \in \{1,2\}$, and negative discriminant}
%\begin{ruledtabular}
\begin{center}
\begin{tabular}{lrr}
Polynomial &
Real root &
Imaginary roots\\
\hline
$x^3-2 x^2-4 x-2$ & $3.36523$ & $-0.68262 \pm 0.35826 i$\\
$x^3-2 x^2-3 x-1$ & $3.07960$ & $-0.53980 \pm 0.18258 i$\\
$x^3-2 x^2-3 x-2$ & $3.15276$ & $-0.57638 \pm 0.54968 i$\\
$x^3-2 x^2-2 x-1$ & $2.83118$ & $-0.41559 \pm 0.42485 i$\\
$x^3-2 x^2-2 x-2$ & $2.91964$ & $-0.45982 \pm 0.68817 i$\\
$x^3-2 x^2-x-1$   & $2.54682$ & $-0.27341 \pm 0.56382 i$\\
$x^3-2 x^2-x-2$   & $2.65897$ & $-0.32948 \pm 0.80225 i$\\
$x^3-2 x^2-1$     & $2.20557$ & $-0.10278 \pm 0.66546 i$\\
$x^3-2 x^2-2$     & $2.35930$ & $-0.17965 \pm 0.90301 i$\\
$x^3-2 x^2+x-1$   & $1.75488$ &  $0.12256 \pm 0.74486 i$\\
$x^3-2 x^2+2 x-2$ & $1.54369$ &  $0.22816 \pm 1.11514 i$\\
$x^3-2 x^2+3 x-1$ & $0.43016$ &  $0.78492 \pm 1.30714 i$\\
$x^3-2 x^2+4 x-1$ & $0.28477$ &  $0.85761 \pm 1.66615 i$\\
$x^3-2 x^2+4 x-2$ & $0.63890$ &  $0.68055 \pm 1.63317 i$\\
$x^3-2 x^2+5 x-1$ & $0.21676$ &  $0.89162 \pm 1.95409 i$\\
$x^3-2 x^2+5 x-2$ & $0.46682$ &  $0.76659 \pm 1.92266 i$
\end{tabular}
\end{center}
%\end{ruledtabular}
\end{table}
From the above argument,
the absolute value of the imaginary part of
$\alpha_2 + \beta_2$ is less than $0.1609$, but there is no
such imaginary root in Table~\ref{tab:b=-1}.
So again we arrive at a contradiction.

Consequently, we have proved the assertion for $c \ge 10$ in Case (ii).
The assertion for $2 \le c < 10$ has already been
confirmed in \cite{SaitoChaos2018}.

By \eqref{eq:RealCubicSymmetry}, proofs for Cases (iii) and (iv)
follow from Cases (ii) and (i), respectively.
\end{proof}

Note that in contrast to the one-parameter families
$\left\{I_{n}^{2,r}\right\}$ and $\left\{I_{n}^{2,i}\right\}$
discussed in the previous sections, the two-parameter family $\left\{
I_{m,n}^{3,ntr} \right\}$ with $m$ unrestricted includes a set which has
more than one element belonging to an identical field.
For example, the two elements $\sqrt[3]{2}-1$ and $\sqrt[3]{2^2}-1$ of
$I_{3,3}^{3,ntr}$ belong to the same cubic field ${\mathbb
Q}(\sqrt[3]{2})$.

We now move on to consider to what extent some unions of
$I_{m,n}^{3,ntr}$ can cover the cubic fields.
For each integer $m$ in $\left\{0, -1, -2, -3\right\}$, we define the
set $S_{m}^{3,ntr}$ by
\[
S_{m}^{3,ntr}:=\bigcup_{n \ge -m+1} I_{m,n}^{3,ntr}.
\]
For the cubic sets $S_{m}^{3,ntr}$ with $m=0$ and $-3$ (i.e.,
$S_{0}^{3,ntr}$ and $S_{-3}^{3,ntr}$), we prove a result weaker than
Theorems \ref{Thm:RealQuadraticCover},
\ref{Thm:RealQuadraticVarietyI}, \ref{Thm:ImaginaryQuadraticCover},
and \ref{Thm:ImaginaryQuadraticVarietyI} for the quadratic
counterparts.
That is, we just show their full generation with respect to real cubic
fields which are not totally real.

\begin{thm}\label{Thm:RealCubicCover}
The set $S_{0}^{3,ntr}$ includes a primitive element of any real cubic field
which is not totally real.
The same is true of $S_{-3}^{3,ntr}$.
\end{thm}

\begin{proof}
By \eqref{eq:RealCubicSymmetry}, it suffices to prove the statement
for $S_{0}^{3,ntr}$.
Let $K$ be a real cubic field with two complex embeddings.
Let $O_K$ be the ring of algebraic integers of $K$.
We denote by $\sigma_i$, $i=1,2,3$, the three embeddings of $K$ into
${\mathbb C}$, where we may assume that $\sigma_1$ is the identity
map on $K$.
Let $\phi$ be the embedding of $K$ into ${\mathbb R} \times {\mathbb
  C}$ defined by $\phi(\alpha) := \left( \sigma_1(\alpha),
\sigma_2(\alpha)\right)$ for $\alpha \in K$.
For a real number $\lambda > 0$, we define the set $D(\lambda) \subset
{\mathbb R} \times {\mathbb C}$ by
\[
D(\lambda) := \left\{(u, s + t i) \in {\mathbb R} \times {\mathbb C} \mid \lvert u \rvert, \lvert s \rvert \le 1/4, \lvert t \rvert \le \lambda \right\}.
\]
This $D(\lambda)$ is a convex set which is symmetric with respect to
the origin.
We choose $\lambda$ sufficiently large so that there exists $\alpha
\in O_K$ such that $\alpha \neq 0$ and $\phi (\alpha) \in D(\lambda)$,
by Minkowski's theorem \cite{Hardy}.
Due to the symmetry of $D(\lambda)$, we can assume that $\alpha > 0$.
Note that $\alpha$ is not a rational number (i.e., $\alpha \notin
{\mathbb Z}$) since $\lvert \alpha \rvert \le 1/4$.
We denote the minimal polynomial of $\alpha$ by $X^3+bX^2+cX+d$.
As before, we put $\alpha_1 := \alpha$, $\alpha_2 :=
\sigma_2(\alpha)$, and $\alpha_3 := \sigma_3(\alpha)$.
Since $\lvert \alpha_1 \rvert \le 1/4$, we see that $0 < \alpha_1 \le
1/4$.
Likewise, we have $\lvert\Re(\alpha_2)\rvert \le 1/4$.
Thus,
\[
\lvert \alpha_1 + \alpha_2 + \alpha_3 \rvert =
\lvert \alpha_1 + 2\Re(\alpha_2) \rvert < 1.
\]
Since $\alpha_1 + \alpha_2 + \alpha_3 \in {\mathbb Z}$, we see that
$\alpha_1 + \alpha_2 + \alpha_3 =0$, which implies $b=0$.
We also see that $\alpha_1 \alpha_2 \alpha_3 \in {\mathbb Z}$ is
nonzero.
This means that $\lvert \alpha_1 \alpha_2 \alpha_3 \rvert \ge 1$, and
therefore $\alpha_2 \alpha_3 \ge 4$.
Then, we have
\begin{align*}
c &= \alpha_1 \alpha_2 + \alpha_2 \alpha_3+ \alpha_3 \alpha_1
= \alpha_1 \left(2\Re(\alpha_2)\right) + \alpha_2 \alpha_3
\ge 4 -1/8 > 0.
\end{align*}
Let $f:{\mathbb R} \rightarrow {\mathbb R}$ be a function given by
$f(x)=x^3+bx^2+cx+d$.
Since $b=0$ and $c > 0$, we see that $f$ is strictly monotonically
increasing.
Since $f(\alpha)=0$ and $0 < \alpha < 1$, we have $f(0)<0$ and
$f(1)>0$, which implies $-c-1<d<0$.
Thus, we have $\alpha \in I_{0,c}^{3,ntr} \subset S_{0}^{3,ntr}$.
\end{proof}

In contrast with $S_{0}^{3,ntr}$ and $S_{-3}^{3,ntr}$, the other sets,
$S_{-1}^{3,ntr}$ and $S_{-2}^{3,ntr}$, are not full generative with
respect to real cubic fields which are not totally real.
For example, any algebraic integer in ${\mathbb Q}(\sqrt[3]{2})$ has a
multiple of three as its trace, and therefore it does not belong to
$S_{-1}^{3,ntr}$ and $S_{-2}^{3,ntr}$, which consist of algebraic integers
whose traces are one and two, respectively.

\section{Properties of $I_{m,n}^{3,tr}$}

This section deals with the sets $I_{m,n}^{3,tr}$, which are defined
in \eqref{eq:TotallyRealCubicDefinition}.
By definition, $I_{m,n}^{3,tr}$
consists of real algebraic integers of degree at most three.
Since the open unit interval $(0,1)$ contains no integer, any $\alpha
\in I_{m,n}^{3,tr}$ is irrational.
Let $f(x)=x^3+mx^2+nx+d$ for $m,n,d \in {\mathbb Z}$ satisfying $n \le
-m-3$ and $1 \le d \le -m-n-2$.
Since $f$ is a monic cubic polynomial satisfying $f(0)>0$ and
$f(1)<0$, $f$ has three real roots, only one of which is in $(0,1)$.
Thus, if $\alpha \in I_{m,n}^{3,tr}$ is cubic, then $\alpha$ is
totally real, that is, all conjugates of $\alpha$ are real.
We will show that $I_{m,n}^{3,tr}$ contains at most one, if any,
quadratic integer, and all other elements are totally real cubic
(cf. \thmref{Thm:TotallyRealCubicNoncoincidence}).

As in \eqref{eq:RealCubicSymmetry}, we have
\begin{align}\label{eq:TotallyRealCubicSymmetry}
I_{-b-3, 2b + c +3}^{3,tr} &= \left\{1-\alpha \mid \alpha \in I_{b,c}^{3,tr}\right\}.
\end{align}

We obtain the following result concerning the uniformity of
$I_{b,c}^{3,tr}$.

\begin{thm}\label{Thm:TotallyRealCubicUniformity}
Let $b$ be a fixed integer.
Then, the elements of $I_{b,c}^{3,tr}$ are distributed almost
uniformly in the unit interval for sufficiently large negative $c$.
\end{thm}

\begin{proof}
We first introduce some notation.
Let $b,c,d$ be integers satisfying $c \le -b-3$ and $1 \le d \le
-b-c-2$.
Let $f_{d}(x)=x^3+bx^2+cx+d$, and let $\alpha_{d}$ be the root of
$f_{d}$ in the open unit interval $(0,1)$.
We see easily that $\alpha_{d-1} < \alpha_{d}$ and
$f_{d}(\alpha_{d-1})=1$ hold for $d \in
\left\{2,3,\dots,-b-c-2\right\}$.
Let $\Delta_{d} = \alpha_{d} - \alpha_{d-1}$ ($d \in
\left\{2,3,\dots,-b-c-2\right\}$).
By the mean value theorem, there exists $\beta \in \left(\alpha_{d-1},
\alpha_{d} \right)$ such that $f_{d}'(\beta)=-\Delta_{d}^{-1}$.

By \eqref{eq:TotallyRealCubicSymmetry}, it suffices to show the
assertion for $b \ge -1$.
We consider two cases.
If $b = -1$, we easily verify that $c-1/3 \le f_{d}'(x) < c+1$
holds for $x \in \left(0, 1 \right)$.
This gives
\[
\left(-c+1/3\right)^{-1} \le \Delta_{d} < \left(-c-1\right)^{-1}.
\]
We denote by $N$ the number of the elements of $I_{-1,c}^{3,tr}$,
namely $N = -c-1$.
Then, we have
\[
\frac{-4/3}{N(N + 4/3)} = \frac{1}{N + 4/3} -\frac{1}{N} \le \Delta_{d} -\frac{1}{N}< \frac{1}{N} -\frac{1}{N} =0,
\]
which implies
\[
\max_{d \in \left\{2,3,\dots, N\right\}} \left| \Delta_{d} -\frac{1}{N} \right|
\le \frac{4/3}{N(N + 4/3)}.
\]
Thus, for sufficiently large negative $c$ (and therefore sufficiently
large $N$), the elements of $I_{-1,c}^{3,tr}$ are distributed almost
uniformly in the unit interval.

If $b \ge 0$, we see that $c < f_{d}'(x) < 3 + 2 b + c$ holds for $x
\in \left(0, 1 \right)$.
Suppose that $c < -2 b -3$.
Then, we have
\[
\left(-c\right)^{-1} < \Delta_{d} < \left(-3 - 2 b - c\right)^{-1}.
\]
As above, we denote by $N$ the number of the elements of
$I_{b,c}^{3,tr}$, i.e., $N = -b-c-2$.
Then, we have
\[
\frac{-b-2}{N(N + b + 2)} = \frac{1}{N + b + 2} -\frac{1}{N} < \Delta_{d} -\frac{1}{N}< \frac{1}{N -b -1} -\frac{1}{N} = \frac{b+1}{N(N - b -1)},
\]
which implies
\[
\max_{d \in \left\{2,3,\dots, N\right\}} \left| \Delta_{d} -\frac{1}{N} \right|
< \max \left\{ \frac{b+2}{N(N + b + 2)}, \frac{b+1}{N(N - b -1)} \right\}.
\]
Thus, for sufficiently large negative $c$ (and therefore sufficiently
large $N$), the elements of $I_{b,c}^{3,tr}$ are distributed almost
uniformly in the unit interval.
\end{proof}

Moreover, for any $I_{m,n}^{3,tr}$ with $m \in \left\{0, -1, -2,
-3\right\}$ and $n \le -m-3$, we can show that its elements are
divided into the intervals $(0,1/2)$ and $(1/2,1)$ as equally as
possible, similarly to $I_{n}^{2,r}$ ($n \ge 1$), $I_{n}^{2,i}$ ($n
\ge 1$), and $I_{m,n}^{3,ntr}$ ($m \in \left\{0, -1, -2, -3\right\}$, $n
\ge -m+1$) (see discussion following \propref{Prop:CubicUniformity}).

We now establish the arithmetical independence
of $I_{m,n}^{3,tr}$ with $m \in \left\{0, -1, -2, -3\right\}$.
In addition, we show that such $I_{m,n}^{3,tr}$ contains at most one, if
any, real quadratic integer, and all other elements are totally real
cubic.

\begin{thm}\label{Thm:TotallyRealCubicNoncoincidence}
Let $b, c$ be integers satisfying $-3 \le b \le 0$ and $c \le -b-3$.
Then, ${\mathbb Q}(\alpha) \neq {\mathbb Q}(\beta)$ holds for all
$\alpha, \beta \in I_{b,c}^{3,tr}$ with $\alpha \neq \beta$.
Moreover, all the elements of $I_{b,c}^{3,tr}$ are totally real cubic
if $c$ is of the form $-n^2 +(b-1) n -1$ ($n \in {\mathbb
  Z}$, $n \ge \delta_b$) or $-n^2 + b n$ ($n \in {\mathbb Z}$, $n \ge 1 +
\epsilon_b -\epsilon_{-3-b}$), where $\delta_b, \epsilon_b$ ($b \in
\left\{0, -1, -2, -3\right\}$) are integers defined by $\delta_{0}
=\delta_{-1} = 1$, $\delta_{-2} =\delta_{-3} = 0$, $\epsilon_{0} =
1$, $\epsilon_{-1} =\epsilon_{-2} =\epsilon_{-3} = 0$.
Otherwise all the elements, except one, are totally real cubic.
The exception is the element that is real quadratic.
Specifically, for $c$ satisfying $-n^2 +(b-1) n -1 < c < -n^2 + b n$
($n \in {\mathbb Z}$, $n \ge 1 + \epsilon_b$), the unique quadratic
element is a root of the polynomial $x^3 +b x^2 + c x - (n-b) (c + n^2
-b n)$, and its minimal polynomial is given by $x^2 + n x + c + n^2 -b
n$.
The same is true for $c$ satisfying $-n^2 + b n < c < -n^2 +(b-1) n
-1$ ($n \in {\mathbb Z}$, $n \le -3 -\epsilon_{-3-b}$).
\end{thm}

Note that every integer $c \le -b-3$ falls into exactly one of the
four cases stated in the statement, namely $c = -n^2 +(b-1) n -1$ ($n
\in {\mathbb Z}$, $n \ge \delta_b$), $c = -n^2 + b n$ ($n \in {\mathbb
  Z}$, $n \ge 1 + \epsilon_b -\epsilon_{-3-b}$), $-n^2 +(b-1) n -1 < c
< -n^2 + b n$ ($n \in {\mathbb Z}$, $n \ge 1 + \epsilon_b$), and $-n^2
+ b n < c < -n^2 +(b-1) n -1$ ($n \in {\mathbb Z}$, $n \le -3
-\epsilon_{-3-b}$).

\begin{proof}
We begin with the case $b=0$ with $c = -n^2 -n -1$ ($n \in {\mathbb
  Z}$, $n \ge 1$).
Let $f(x)=x^3+cx+d$ with $d \in \{1,2,\dots,-c-2\}$.
We see that $f$ has three real roots, since $f(0)>0$, $f(1)<0$,
$f(-n-1)<0$, $f(-n)>0$, $f(n)<0$, and $f(n+1)>0$ hold.
More specifically, each of three open intervals $(0,1)$,
$(-n-1,-n)$, and $(n,n+1)$ contains only one of the three roots.
This, together with the fact that $f$ is a monic cubic polynomial with
integer coefficients, implies that $f$ is irreducible, and thus all the
elements of $I_{0,c}^{3,tr}$ are totally real cubic.

We now assume that ${\mathbb Q}(\alpha) = {\mathbb Q}(\beta)$ holds
for some $\alpha, \beta \in I_{0,c}^{3,tr}$ with $\alpha \neq \beta$.
Let $\alpha_i$ ($i \in \{1,2,3\}$) (resp. $\beta_i$ ($i \in \{1,2,3\}$))
be the conjugates of $\alpha$ (resp. $\beta$).
As in the proof of \thmref{Thm:CubicNoncoincidence},
we assume that $\alpha_1 = \alpha$, $\beta_1 = \beta$, and
${\mathbb Q}(\alpha_i)={\mathbb Q}(\beta_i)$ ($i \in \{1,2,3\}$).
Note that $\alpha_i$ and $\beta_i$ with $i > 1$ are in either
$(n,n+1)$ or $(-n-1,-n)$.
We assume that $\alpha_2$ is in $(n,n+1)$ ($\beta_2$ can be in
$(n,n+1)$ or $(-n-1,-n)$).

First assume that $\beta_2 \in (n,n+1)$.
Since $\lvert \alpha_1 - \beta_1 \rvert < 1$, $\lvert \alpha_2 -
\beta_2 \rvert < 1$, and $\lvert \alpha_3 - \beta_3 \rvert < 1$, we
have $\lvert (\alpha_1 - \beta_1)(\alpha_2 - \beta_2)(\alpha_3 -
\beta_3) \rvert < 1$.
Since $(\alpha_1 - \beta_1)(\alpha_2 - \beta_2)(\alpha_3 - \beta_3)
\in {\mathbb Z}$, we see that $(\alpha_1 - \beta_1)(\alpha_2 -
\beta_2)(\alpha_3 - \beta_3) =0$.
This contradicts $\alpha \neq \beta$.

Now assume that $\beta_2 \in (-n-1,-n)$.
Then, we have $0< \alpha_1 + \beta_1 < 2$, $-1 < \alpha_2 + \beta_2 <
1$, and $-1 < \alpha_3 + \beta_3 < 1$, which give $-2<(\alpha_1 +
\beta_1)(\alpha_2 + \beta_2)(\alpha_3 + \beta_3)<2$.
Since $(\alpha_1 + \beta_1)(\alpha_2 + \beta_2)(\alpha_3 + \beta_3)
\in {\mathbb Z}$, we see that $(\alpha_1 + \beta_1)(\alpha_2 +
\beta_2)(\alpha_3 + \beta_3) \in \{-1,0,1\}$.
We now show that $\alpha_1 + \beta_1 \notin {\mathbb Q}$ (i.e.,
$\alpha_1 + \beta_1 \notin {\mathbb Z}$).
Assuming $r:=\alpha_1 + \beta_1 \in {\mathbb Q}$, we have $3 r=
(\alpha_1 + \beta_1)+(\alpha_2 + \beta_2)+(\alpha_3 + \beta_3)=0$,
which contradicts $r \neq 0$.
Thus, $\alpha_1 + \beta_1$ is a cubic algebraic integer.
We have
\[
-5<(\alpha_1 + \beta_1)(\alpha_2 + \beta_2)+ (\alpha_2 + \beta_2)(\alpha_3 + \beta_3)+ (\alpha_3 + \beta_3)(\alpha_1 + \beta_1)<5.
\]
Therefore, the minimal polynomial of $\alpha_1 + \beta_1$, denoted by
$x^3 + p x - q$, satisfies $\lvert p \rvert \le 4$ and $q \in
\{-1,0,1\}$.
Table~\ref{tab:TotallyRealb=0pr1} shows numerical roots of $x^3 + p x
- q$ which is irreducible and whose discriminant is positive.
\begin{table}
\caption{\label{tab:TotallyRealb=0pr1}
Numerical roots of irreducible polynomials $x^3 + p x - q$ with $\lvert p \rvert \le 4$, $q \in \{-1,0,1\}$, and positive discriminant}
%\begin{ruledtabular}
\begin{center}
\begin{tabular}{lrrr}
Polynomial &
 &
Roots &
\\
\hline
$x^3-4 x+1$ & $-2.11491$ &  $0.25410$ & $1.86081$\\
$x^3-4 x-1$ & $-1.86081$ & $-0.25410$ & $2.11491$\\
$x^3-3 x+1$ & $-1.87939$ &  $0.34730$ & $1.53209$\\
$x^3-3 x-1$ & $-1.53209$ & $-0.34730$ & $1.87939$
\end{tabular}
\end{center}
%\end{ruledtabular}
\end{table}
From the above argument, all the roots of the minimal polynomial of
$\alpha_1 + \beta_1$ are greater than $-1$, but every polynomial in
Table~\ref{tab:TotallyRealb=0pr1} has a root less than $-1$.
So again we arrive at a contradiction.

The proof for the case $b=0$ with $c = -n^2$ ($n \in {\mathbb Z}$, $n \ge 2$)
can be obtained by a similar argument.

We now consider the case $b=0$ with $-n^2 -n -1 < c < -n^2$ ($n \in {\mathbb Z}$,
$n \ge 2$).
We can see that $f$ satisfies $f(0)>0$, $f(1)<0$, $f(-n-1)<0$,
$f(-n)>0$, $f(n-1/2)<0$, and $f(n+1/2)>0$.
Thus, $f$ has three real roots, each of which is contained in one of
the three intervals $(0,1)$, $(-n-1,-n)$, and $(n-1/2,n+1/2)$.
We now show that $I_{0,c}^{3,tr}$ contains only one quadratic element
and that all other elements are cubic.
In order to clarify the dependence of $f$ on $d$, we denote
$x^3+cx+d$ by $f_d(x)$ below.
Letting $i = -c - n^2$, we find $i \in \{1,2,\dots,n\}$ and $n i \in
\{1,2,\dots,-c-2\}$.
Since $x^3 + c x + n i = (x -n) (x^2 + n x -i)$, the root of $f_d$
with $d = n i$ in the interval $(0,1)$ is a quadratic algebraic
integer, whose minimal polynomial is given by $x^2 + n x -i$.
Let $\gamma(d)$ be the root of $f_d$ in the interval $(n-1/2,n+1/2)$.
We see easily that $\gamma(d)$ satisfies $\gamma(1) > \gamma(2) > \dots
> \gamma({-c-2})$.
Since $n$ is only one integer in the interval $(n-1/2,n+1/2)$, $f_d$
is reducible if and only if $d = n i$.
Thus, the above quadratic element is the only quadratic element in
$I_{0,c}^{3,tr}$, and all the other elements are cubic algebraic integers.

We denote, for now, the root of $f_d$ in the interval $(0,1)$ by
$\alpha(d)$ ($d \in \{1,2,\dots,-c-2\}$).
As noted above, $\alpha(n i)$ is the only quadratic element in
$I_{0,c}^{3,tr}$, and thus ${\mathbb Q}(\alpha(d)) \neq {\mathbb
  Q}(\alpha(n i))$ holds for all $d \neq n i$.
Hence it suffices to show that ${\mathbb Q}(\alpha(d)) \neq {\mathbb
  Q}(\alpha(d'))$ holds for all $d, d' \in \{1,2,\dots,-c-2\}
\setminus \{n i\}$ with $d \neq d'$.
We put $\alpha = \alpha(d)$ and $\beta = \alpha(d')$ and assume that
${\mathbb Q}(\alpha) = {\mathbb Q}(\beta)$, as before.
We let $\alpha_i$ ($i \in \{1,2,3\}$) (resp. $\beta_i$ ($i \in
\{1,2,3\}$)) be the conjugates of $\alpha$ (resp. $\beta$) and let
$\alpha_1 = \alpha$, $\beta_1 = \beta$, and ${\mathbb
  Q}(\alpha_i)={\mathbb Q}(\beta_i)$ ($i \in \{1,2,3\}$).
Note that $\alpha_i$ and $\beta_i$ with $i > 1$ are in either
$(n-1/2,n+1/2)$ or $(-n-1,-n)$.
We assume that $\alpha_2$ is in $(n-1/2,n+1/2)$ ($\beta_2$ can be in
$(n-1/2,n+1/2)$ or $(-n-1,-n)$).

First assume that $\beta_2 \in (n-1/2,n+1/2)$.
Since $\lvert \alpha_1 - \beta_1 \rvert < 1$, $\lvert \alpha_2 -
\beta_2 \rvert < 1$, and $\lvert \alpha_3 - \beta_3 \rvert < 1$, we
have $\lvert (\alpha_1 - \beta_1)(\alpha_2 - \beta_2)(\alpha_3 -
\beta_3) \rvert < 1$.
Since $(\alpha_1 - \beta_1)(\alpha_2 - \beta_2)(\alpha_3 - \beta_3)
\in {\mathbb Z}$, we see that $(\alpha_1 - \beta_1)(\alpha_2 -
\beta_2)(\alpha_3 - \beta_3) =0$.
This contradicts $\alpha \neq \beta$.

Now assume that $\beta_2 \in (-n-1,-n)$.
Then, we have $0< \alpha_1 + \beta_1 < 2$, $-3/2 < \alpha_2 + \beta_2 <
1/2$, and $-3/2 < \alpha_3 + \beta_3 < 1/2$, which give $-3/2<(\alpha_1 +
\beta_1)(\alpha_2 + \beta_2)(\alpha_3 + \beta_3)<9/2$.
Since $(\alpha_1 + \beta_1)(\alpha_2 + \beta_2)(\alpha_3 + \beta_3)
\in {\mathbb Z}$, we see that $(\alpha_1 + \beta_1)(\alpha_2 +
\beta_2)(\alpha_3 + \beta_3) \in \{-1,0,1,2,3,4\}$.
We now show that $\alpha_1 + \beta_1 \notin {\mathbb Q}$ (i.e.,
$\alpha_1 + \beta_1 \notin {\mathbb Z}$).
Assuming $r:=\alpha_1 + \beta_1 \in {\mathbb Q}$, we have $3 r=
(\alpha_1 + \beta_1)+(\alpha_2 + \beta_2)+(\alpha_3 + \beta_3)=0$,
which contradicts $r \neq 0$.
Thus, $\alpha_1 + \beta_1$ is a cubic algebraic integer.
We have
\[
-\frac{27}{4}<(\alpha_1 + \beta_1)(\alpha_2 + \beta_2)+ (\alpha_2 + \beta_2)(\alpha_3 + \beta_3)+ (\alpha_3 + \beta_3)(\alpha_1 + \beta_1)<\frac{17}{4}.
\]
Therefore, the minimal polynomial of $\alpha_1 + \beta_1$, denoted by
$x^3 + p x - q$, satisfies $p \in \{-6,-5,\dots,3,4\}$ and $q \in
\{-1,0,1,2,3,4\}$.
Table~\ref{tab:TotallyRealb=0pr2} shows numerical roots of $x^3 + p x
- q$ which is irreducible and whose discriminant is positive.
\begin{table}
\caption{\label{tab:TotallyRealb=0pr2}
Numerical roots of irreducible polynomials $x^3 + p x - q$ with $p \in \{-6,-5,\dots,3,4\}$, $q \in \{-1,0,1,2,3,4\}$, and positive discriminant}
%\begin{ruledtabular}
\begin{center}
\begin{tabular}{lrrr}
Polynomial &
&
Roots &
\\
\hline
$x^3-6 x+1$ & $-2.52892$ &  $0.16745$ & $2.36147$\\
$x^3-6 x-1$ & $-2.36147$ & $-0.16745$ & $2.52892$\\
$x^3-6 x-2$ & $-2.26180$ & $-0.33988$ & $2.60168$\\
$x^3-6 x-3$ & $-2.14510$ & $-0.52398$ & $2.66908$\\
$x^3-5 x+1$ & $-2.33006$ &  $0.20164$ & $2.12842$\\
$x^3-5 x-1$ & $-2.12842$ & $-0.20164$ & $2.33006$\\
$x^3-5 x-3$ & $-1.83424$ & $-0.65662$ & $2.49086$\\
$x^3-4 x+1$ & $-2.11491$ &  $0.25410$ & $1.86081$\\
$x^3-4 x-1$ & $-1.86081$ & $-0.25410$ & $2.11491$\\
$x^3-4 x-2$ & $-1.67513$ & $-0.53919$ & $2.21432$\\
$x^3-3 x+1$ & $-1.87939$ &  $0.34730$ & $1.53209$\\
$x^3-3 x-1$ & $-1.53209$ & $-0.34730$ & $1.87939$
\end{tabular}
\end{center}
%\end{ruledtabular}
\end{table}
From the above argument, all the roots of the minimal polynomial of
$\alpha_1 + \beta_1$ are greater than $-3/2$, but every polynomial in
Table~\ref{tab:TotallyRealb=0pr2} has a root less than $-3/2$.
So again we arrive at a contradiction.

The proof for the case $b=0$ with $-n^2 < c < -n^2 -n -1$ ($n \in {\mathbb Z}$,
$n \le -3$) can be obtained by a similar argument.

Thus we have completed the proof for the case $b=0$.

The proof for the case $b=-1$ is obtained from an argument analogous to that used
in the case $b=0$.

By \eqref{eq:TotallyRealCubicSymmetry} together with discussions
succeeding \eqref{eq:RealQuadraticSymmetry} and
preceding \eqref{eq:RealCubicSymmetry}, proofs for the cases $b=-2$ and $b=-3$
follow from the cases $b=-1$ and $b=0$, respectively.
\end{proof}

For each integer $m$ in $\left\{0, -1, -2, -3\right\}$, we define the
set $S_{m}^{3,tr}$ by
\[
S_{m}^{3,tr}:=\bigcup_{n \le -m-3} I_{m,n}^{3,tr}.
\]
We denote by ${\rm Quad}(S)$ the set of all quadratic elements of a
given set $S$.
From the above theorem and the discussion at the beginning of the
proof of \thmref{Thm:RealQuadraticCover}, we obtain:

\begin{cor}\label{Cor:TotallyRealCubicRealQuadraticCover}
The followings hold:
\begin{enumerate}[(i)]
\item The set of all quadratic algebraic integers in $S_{0}^{3,tr}$
(resp. $S_{-1}^{3,tr}$, $S_{-2}^{3,tr}$, and $S_{-3}^{3,tr}$)
is expressed as
\begin{align*}
  {\rm Quad}(S_{0}^{3,tr}) &= \bigcup_{n \in {\mathbb Z} \setminus \left\{-2, -1, 0, 1\right\}} I_{n}^{2,r} = \left(S^{2,r} \setminus \left\{ \frac{-1 + \sqrt{5}}{2}\right\}\right) \bigcup \left(-S^{2,r} + 1 \right),\\
  {\rm Quad}(S_{-1}^{3,tr}) &= {\rm Quad}(S_{-2}^{3,tr})=\bigcup_{n \in {\mathbb Z} \setminus \left\{-2, -1, 0\right\} } I_{n}^{2,r} = S^{2,r} \bigcup \left(-S^{2,r} + 1 \right),\\
  {\rm Quad}(S_{-3}^{3,tr}) &= \bigcup_{n \in {\mathbb Z} \setminus \left\{-3, -2, -1, 0\right\}} I_{n}^{2,r} = S^{2,r} \bigcup \left(\left(-S^{2,r} + 1 \right) \setminus \left\{ \frac{3 - \sqrt{5}}{2}\right\}\right),
\end{align*}
where $S^{2,r}$ is given by \eqref{eq:S2r} in Section
\ref{sec:PropertiesofI_{n}^{2,r}}.

\item In particular, $S_{m}^{3,tr}$ with $m \in \left\{0, -1, -2,
  -3\right\}$ includes a primitive element of any real quadratic
  field.
\end{enumerate}
\end{cor}

To rephrase \cororef{Cor:TotallyRealCubicRealQuadraticCover} (ii),
$S_{m}^{3,tr}$ with $m \in \left\{0, -1, -2, -3\right\}$ are full
generative with respect to real quadratic fields.
Moreover, analogously to Theorem \ref{Thm:RealCubicCover}, we show that
$S_{m}^{3,tr}$ with $m=0$ and $-3$ (i.e., $S_{0}^{3,tr}$ and
$S_{-3}^{3,tr}$) are full generative with respect to totally real
cubic fields:
\begin{thm}\label{Thm:TotallyRealCubicCover}
The set $S_{0}^{3,tr}$ includes a primitive element of any totally real cubic field.
The same is true of $S_{-3}^{3,tr}$.
\end{thm}

\begin{proof}
By \eqref{eq:TotallyRealCubicSymmetry}, it suffices to prove the statement
for $S_{0}^{3,tr}$.
Let $K$ be a totally real cubic field. 
Let $O_K$ be the ring of algebraic integers of $K$.
We denote by $\sigma_i$, $i=1,2,3$, the three embeddings of $K$ into
${\mathbb C}$, where we may assume that $\sigma_1$ is the identity
map on $K$.
Let $\phi$ be the embedding of $K$ into ${\mathbb R}^3$ defined by $\phi(\alpha) := \left( \sigma_1(\alpha),
\sigma_2(\alpha), \sigma_3(\alpha)\right)$ for $\alpha \in K$.
For positive real numbers $\lambda_1$, $\lambda_2$, and $\lambda_3$,
we define the set $D(\lambda_1,\lambda_2,\lambda_3) \subset {\mathbb
R}^3$ by
\[
D(\lambda_1,\lambda_2,\lambda_3) := \left\{(x_1,x_2,x_3) \in {\mathbb R}^3 \mid \lvert x_1 \rvert\le \lambda_1, \lvert x_2 \rvert, \lvert x_3 \rvert \le \lambda_2, \lvert x_2+x_3 \rvert \le \lambda_3 \right\}.
\]
This set is convex and symmetric with respect to the origin.
If $\sigma_2(\alpha)+\sigma_3(\alpha)=0$ holds for $\alpha \in K$,
then $\alpha=0$ since the trace of $\alpha$ is zero.
Therefore, there exists $0<\epsilon<1/2$ such that
\begin{align}\label{a}
D\left(1/4,3+\epsilon,\epsilon\right) \cap \phi(O_K)=\left\{\phi(0)\right\}.
\end{align}
We choose $\lambda>0$ sufficiently large so that there exists $\beta
\in O_K$ such that $\beta \neq 0$ and $\phi (\beta) \in D(1/4,\lambda,\epsilon)$ by Minkowski's theorem.
Due to the symmetry of $D(1/4,\lambda,\epsilon)$, we can assume that $\beta > 0$.
Note that $\beta$ is not a rational number (i.e., $\beta \notin
{\mathbb Z}$) since $0<\epsilon<1/2$.
We denote the minimal polynomial of $\beta$ by $X^3+bX^2+cX+d$.
We put $\beta_1 := \beta$, $\beta_2 :=
\sigma_2(\beta)$, and $\beta_3 := \sigma_3(\beta)$.
Then,
\begin{align*}
&\lvert \beta_1 + \beta_2 + \beta_3 \rvert 
\leq \lvert \beta_1\rvert + \lvert \beta_2 + \beta_3\rvert
< \dfrac{1}{4}+ \dfrac{1}{2}<1.
\end{align*}
Since $\beta_1 + \beta_2 + \beta_3 \in {\mathbb Z}$, we see that
$\beta_1 + \beta_2 + \beta_3 =0$, which implies $b=0$.
From \eqref{a} we have $|\beta_2|, |\beta_3| > 3$,
which, together with $\lvert \beta_2+\beta_3\rvert < 1/2$, gives
$\beta_2\beta_3 < -9$.
Then, we have
\begin{align*}
c=\beta_1\beta_2+\beta_2\beta_3+\beta_3\beta_1=\beta_1(\beta_2+\beta_3)+\beta_2\beta_3<\frac{1}{8}-9<-8,
\end{align*}
and
\begin{align*}
1+c+d=1 + \beta_1(\beta_2+\beta_3)+\beta_2\beta_3(1-\beta_1)
<1 + \frac{1}{8}-9 \cdot \frac{3}{4}<0.
\end{align*}
Since $\beta_1 > 0$ and $\beta_2\beta_3 < -9$, we see that $d >0$.
Therefore, we have $\beta \in S_{0}^{3,tr}$.
\end{proof}

Similarly to $S_{-1}^{3,ntr}$ and $S_{-2}^{3,ntr}$ discussed in the
previous section, we note that $S_{-1}^{3,tr}$ and $S_{-2}^{3,tr}$ are
not full generative with respect to totally real cubic fields, in
contrast with $S_{0}^{3,tr}$ and $S_{-3}^{3,tr}$.
For example, consider a totally real cubic field $K={\mathbb
  Q}(\alpha)$, where $\alpha$ is one of the roots of $X^3-3X+1$.
Since any algebraic integer in $K$ has a multiple of three as its
trace, it does not belong to $S_{-1}^{3,tr}$ and $S_{-2}^{3,tr}$.

\section*{Acknowledgements}

We thank Shigeki Akiyama and Zeev Rudnick for their helpful comments.
This research was supported by JSPS KAKENHI Grant Numbers JP15K00342,
JP16KK0005, and JP22K12197.

\end{document}